\documentclass{amsart}
\usepackage{amssymb}
\usepackage{hyperref}
\usepackage[a4paper]{geometry}
\usepackage{dsfont}
\usepackage{todonotes}
\theoremstyle{plain}
\newtheorem{thm}{Theorem}[section] 
 
\newtheorem{lem}[thm]{Lemma} 
\newtheorem{cor}[thm]{Corollary}
\theoremstyle{remark} 
\newtheorem*{rema}{Remark}
\numberwithin{equation}{section}
  \usepackage{datetime}

\def\R{\mathbb R}

\def\val#1{\vert#1\vert}

\def\sign{\operatorname{sign}}
\begin{document}
\title{An energy method for averaging lemmas}
% ============
% = author 1 =
% ============
\author{Diogo Ars\'enio} 
\address { New York University Abu Dhabi \\
Abu Dhabi \\
United Arab Emirates} 
\email{diogo.arsenio@nyu.edu}
% ============
% = author 2 =
% ============
\author{Nicolas Lerner}
\address {Institut de Math\'ematiques de Jussieu,
Sorbonne Universit\'e \\
Campus Pierre et Marie Curie,
4 Place Jussieu,
75252 Paris cedex 05 \\
France} 
\email{nicolas.lerner@imj-prg.fr}
%
% \subjclass{}
\date{\today}

\begin{abstract}
	This work introduces a new approach to velocity averaging lemmas in kinetic theory. This approach---based upon the classical energy method---provides a powerful duality principle in kinetic transport equations which allows for a natural extension of classical averaging lemmas to previously unknown cases where the density and the source term belong to dual spaces. More generally, this kinetic duality principle produces regularity results where one can trade a loss of regularity or integrability somewhere in the kinetic transport equation for a suitable opposite gain elsewhere.
	Also, it looks simpler and more robust to rely on proving inequalities instead of constructing exact parametrices.
	\par
	The results in this article are introduced from a functional analytic point of view. They are motivated by the abstract regularity theory of kinetic transport equations. However, we may recall that velocity averaging lemmas have profound implications in kinetic theory and its related physical models. In particular, the precise formulation of such results has the potential to lead to important applications to the regularity of renormalizations of Boltzmann-type equations, as well as kinetic formulations of gas dynamics, for instance.
\end{abstract}

\maketitle

% \tableofcontents

% ================
% = Introduction =
% ================
\section{Introduction}

From a purely mathematical point of view, kinetic theory is generally concerned with the study of partial differential equations, or integro-differential equations, describing the evolution of a non-negative density $f(t,x,v)$ of particles, where $t\in\mathbb{R}$ is the time, $x\in\Omega\subset\mathbb{R}^n$ is their location in some domain, and $v\in\mathbb{R}^n$ represents their velocities.

Examples of such equations include the celebrated Boltzmann equation
\begin{equation}\label{boltzmann}
	\partial_tf+v\cdot\nabla_xf=Q(f,f),
\end{equation}
which describes the evolution of a dilute gas. The transport operator in the left-hand side expresses the fact that particles travel at velocity $v$, whereas the Boltzmann operator $Q(f,f)$ in the right-hand side is a quadratic integral operator taking into account the changes in density due to collisions in the gas. The existence of global renormalized solutions to the Boltzmann equation was first established in \cite{dl89:1}. We also refer to \cite{cip94} for an introduction to the subject.

The Vlasov--Maxwell system (with $n=3$)
\begin{equation}\label{VM}
	\left\{
	\begin{aligned}
		& \partial_tf+v\cdot\nabla_xf+(E+v\times B)\cdot\nabla_vf=0,\\
		& \partial_tE-\nabla_x\times B=-\int_{\mathbb{R}^3}vf(t,x,v)dv, & \nabla_x\cdot E&=\int_{\mathbb{R}^3}f(t,x,v)dv,\\
		& \partial_tB+\nabla_x\times E=0, & \nabla_x\cdot B&=0,
	\end{aligned}
	\right.
\end{equation}
which describes the evolution of a plasma whose charged particles are subject to the action of a self-induced electromagnetic field $(E,B)$, is another fundamental example of a kinetic model coming from non-equilibrium statistical mechanics. The existence of global weak solutions to the above system was established in \cite{dl89:2}.

In both examples, the regularity (or compactness) properties of the kinetic transport equation
\begin{equation}\label{transport:2}
	(\partial_t+v\cdot\nabla_x)f(t,x,v)=g(t,x,v),
\end{equation}
or its stationary version
\begin{equation}\label{transport:1}
	v\cdot\nabla_xf(x,v)=g(x,v),
\end{equation}
played a crucial role in showing the existence of global solutions (see \cite{dl89:2,dl89:1} for details).
\par
Observe, though, that for a fixed velocity field, in general, transport equations do not provide any regularization of solutions, because they propagate singularities due to their hyperbolic behavior. However, in the kinetic interpretation of transport equations, the collective behavior of particles is taken into account, which leads to a smoothing effect known as \emph{velocity averaging.} The most essential form of velocity averaging was established in \cite{glps88}, where it was first shown that velocity averages of $f(t,x,v)$ may enjoy some smoothness, even though $f(t,x,v)$ itself is not regular.
\par
More precisely, employing the standard notation $H^s$, with $s\in\mathbb{R}$, to denote the Sobolev space of functions $f$ such that $(1-\Delta)^\frac s2f\in L^2$, it is shown in \cite{glps88} that, whenever $f,g\in L^2(\mathbb{R}_t\times\mathbb{R}^n_x\times\mathbb{R}^n_v)$ solve the kinetic transport equation \eqref{transport:2}, the velocity averages of $f$ satisfy the regularization
\begin{equation*}
	\int_{\mathbb{R}^n}f(t,x,v)\varphi(v)dv\in H^\frac 12(\mathbb{R}_t\times\mathbb{R}^n_x),
\end{equation*}
for any $\varphi\in L^\infty_c(\mathbb{R}^n)$. Similarly, the results from \cite{glps88} also establish that
\begin{equation}\label{average:1}
	\int_{\mathbb{R}^n}f(x,v)\varphi(v)dv\in H^\frac 12(\mathbb{R}^n_x),
\end{equation}
for any $\varphi\in L^\infty_c(\mathbb{R}^n)$, whenever $f,g\in L^2(\mathbb{R}^n_x\times\mathbb{R}^n_v)$ solve the stationary kinetic transport equation \eqref{transport:1}.

It is generally accepted that the stationary version \eqref{transport:1} describes the essential features of kinetic transport. In other words, it is in general possible to deduce a result for the non-stationary equation \eqref{transport:2} from a corresponding result on the stationary version of the equation \eqref{transport:1}, and vice versa (see the appendices of \cite{am14,am19,as11} for some details on how to relate the stationary and non-stationary cases). Therefore, for the sake of simplicity, we are now going to focus mainly on the stationary setting of kinetic transport.

\medskip
Based on previously existing averaging results and counterexamples, it was argued in \cite{a15,am19} that, when $n\geq 2$, the gain of half a derivative on the velocity averages displayed in \eqref{average:1} is the maximal gain one can hope for through velocity averaging (ignoring other regularizing phenomena such as hypoellipticity). This motivated the authors therein to seek other settings where half a derivative is gained on the averages. More precisely, employing the standard notation $W^{s,r}$, with $s\in\mathbb{R}$ and $1< r<\infty$, to denote the Sobolev space of functions $f$ such that $(1-\Delta)^\frac s2f\in L^r$, the following question was addressed:
\begin{center}
	\medskip
	\begin{minipage}{0.9\textwidth}
		For what values of $1< p,q,r<\infty$ does one have
		that $f,v\cdot\nabla_xf\in L^p(\mathbb{R}^n_x;L^q(\mathbb{R}^n_v))$
		implies $\int_{\mathbb{R}^n}f\varphi dv\in W^{s,r}(\mathbb{R}^n_x)$,
		for every $0\leq s<\frac 12$,
		where $\varphi\in L^\infty_c(\mathbb{R}^n)$ is given?
	\end{minipage}
	\medskip
\end{center}
This line of questioning turns out to be surprisingly rich and complex. In fact, many questions remain open and we refer to \cite{a15} for some conjectures. Moreover, the subject seems to bear some resemblance with difficult open problems from harmonic analysis concerning the boundedness of Bochner--Riesz multipliers and Fourier restriction operators, as well as smoothing conjectures for Schr\"odinger and wave equations. But definitive links between velocity averaging and harmonic analysis remain yet to be uncovered.

Some answers were provided in \cite{am19}, where it was shown that \eqref{average:1} holds, for any $\varphi\in L^\infty_c(\mathbb{R}^n)$ and in any dimension $n\geq 2$, whenever
\begin{equation*}
	f, v\cdot\nabla_x f \in L^p(\mathbb{R}^n_x;L^{p'}(\mathbb{R}^n_v)),
\end{equation*}
for some $\frac{2n}{n+1}<p\leq 2$. Furthermore, by building upon the work from \cite{jv04}, it was also established in \cite{am19} that
\begin{equation*}
	\int_{\mathbb{R}^n}f(x,v)\varphi(v)dv\in W^{s,\frac 43}(\mathbb{R}^n_x),
\end{equation*}
for any $\varphi\in L^\infty_c(\mathbb{R}^n)$ and for all $0\leq s<\frac n{4(n-1)}$, whenever
\begin{equation*}
	f, v\cdot\nabla_x f \in L^\frac 43(\mathbb{R}^n_x;L^{2}(\mathbb{R}^n_v)).
\end{equation*}
Observe that this result allows the regularity index $s$ to get arbitrarily close to the value $\frac 12$ when $n=2$ but not if $n\geq 3$, but there seems to be room for improvement (see \cite{a15, am19} for some challenging conjectures).

The one-dimensional case $n=1$ is much easier and a more complete picture was provided in \cite{am19}, as well. More precisely, it was established therein, for any given $1<p<\infty$ and $1<q\leq\infty$, that
\begin{equation*}
	\int_{\mathbb{R}}f(x,v)\varphi(v)dv\in W^{s,p}(\mathbb{R}_x),
\end{equation*}
for any $\varphi\in L^\infty_c(\mathbb{R})$ and for all $0\leq s<1-\frac 1q$, whenever
\begin{equation*}
	f, v\partial_x f \in L^p(\mathbb{R}_x;L^{q}(\mathbb{R}_v)).
\end{equation*}
Observe that the regularity index $s$ reaches beyond the value $\frac 12$ when $q>2$. However, the one-dimensional methods from \cite{am19} fail in higher dimension and, in fact, by virtue of a counterexample from \cite{dp01}, similar results cannot hold when $n\geq 2$.

\medskip
We wish now to further extend the field of maximal velocity averaging by inquiring:
\begin{center}
	\medskip
	\begin{minipage}{0.9\textwidth}
		For what values of $1< p_0,p_1,q_0,q_1,r<\infty$ does one have that $f\in L^{p_0}(\mathbb{R}^n_x;L^{q_0}(\mathbb{R}^n_v))$ and
		$v\cdot\nabla_xf\in L^{p_1}(\mathbb{R}^n_x;L^{q_1}(\mathbb{R}^n_v))$
		implies $\int_{\mathbb{R}^n}f\varphi dv\in \dot W^{s,r}(\mathbb{R}^n_x)$, for every $0\leq s<\frac 12$,
		where $\varphi\in L^\infty_c(\mathbb{R}^n)$ is given?
	\end{minipage}
	\medskip
\end{center}
Recall here that a given suitable function $h$ belongs to the homogeneous Sobolev space $\dot W^{s,r}$ (resp.\ $\dot H^s$) provided $(-\Delta)^\frac s2h$ belongs to $L^r$ (resp.\ $L^2$). This property defines, in general, a semi-norm over tempered distributions which we only use here to provide a regularity estimate. In particular, for simplicity, we refrain from discussing the settings in which such properties define Banach spaces and we refer the reader to \cite{bcd11} for more details on the subject.

In this work, we address the above problem by introducing a new energy identity for the kinetic transport equation in Section \ref{section:2}, which provides a crucial duality principle between the unknown density $f(x,v)$ and the source term $g(x,v)$ in \eqref{transport:1}. This kinetic duality will allow us to trade integrability and regularity between $f(x,v)$ and $g(x,v)$ without affecting the outcome on the smoothness of the velocity average \eqref{average:1}, thereby establishing new velocity averaging lemmas displaying a maximal gain of regularity of one half derivative.

More precisely, we establish in Section \ref{section:2} that
\begin{equation}\label{average:2}
	\int_{\mathbb{R}^n}f(x,v) dv\in \dot H^\frac 12(\mathbb{R}^n_x),
\end{equation}
whenever $f(x,v)$ is compactly supported and
\begin{equation}\label{bounds:1}
	\begin{aligned}
		f & \in L^p(\mathbb{R}_x^n;L^{q}(\mathbb{R}_v^n)),
		\\
		v\cdot\nabla_x f & \in L^{p'}(\mathbb{R}_x^n;L^{q'}(\mathbb{R}_v^n)),
	\end{aligned}
\end{equation}
for some $p,q\in (1,\infty)$. This result obviously includes and generalizes the classical Hilbertian case $p=q=2$.

In fact, it will be clear in Section \ref{section:2} that the energy method can be extended to a variety of dual couples of suitable reflexive Banach spaces. For instance, we will see that the bounds
\begin{equation*}
	\begin{aligned}
		(1-\Delta_x)^\frac a2 (1-\Delta_v)^\frac \alpha 2f & \in L^p(\mathbb{R}_x^n;L^{q}(\mathbb{R}_v^n)),
		\\
		(1-\Delta_x)^{-\frac a2} (1-\Delta_v)^{-\frac \alpha 2}(v\cdot\nabla_x f) & \in L^{p'}(\mathbb{R}_x^n;L^{q'}(\mathbb{R}_v^n)),
	\end{aligned}
\end{equation*}
for some $p,q\in (1,\infty)$ and any $a,\alpha\in\mathbb{R}$, will suffice to deduce \eqref{average:2}, provided $f(x,v)$ is compactly supported in $v$. These results exemplify perfectly how a trade-off between $f$ and $v\cdot\nabla_x f$ in both integrability and regularity can be exploited to preserve the smoothness of velocity averages.

In Section \ref{section:3}, we will delve further into the effects of hypoellipticity and singular source terms on the energy method, by assuming more general a priori regularity bounds on $f$ and $v\cdot\nabla_xf$. A careful analysis will show us how to trade integrability and regularity between $f$ and $v\cdot\nabla_x f$ in order to preserve the smoothness of velocity averages deduced from the corresponding Hilbertian case. Precise statements of such results are found in Corollaries \ref{result:5}, \ref{result:4}, \ref{result:6} and \ref{result:7}, below.\footnote{Observe, in the statements of Corollaries \ref{result:5}, \ref{result:4}, \ref{result:6} and \ref{result:7}, how the regularity index $\sigma$ is independent of the integrability parameters $p$ and $q$, thereby showing that the smoothness of the Hilbertian case $p=q=2$ is preserved when trading integrability.}

Finally, in Section \ref{section:dispersion}, in order to further illustrate the robustness and versatility of the energy method, we will explore how kinetic dispersive phenomena can provide new averaging lemmas when combined with the energy method. Specifically, we investigate therein how the bounds
\begin{equation*}
	\begin{aligned}
		f & \in L^{r_0}(\mathbb{R}_x^n;L^{p_0}(\mathbb{R}_v^n)),
		\\
		v\cdot\nabla_x f & \in L^{r_1}(\mathbb{R}_x^n;L^{p_1}(\mathbb{R}_v^n)),
	\end{aligned}
\end{equation*}
for some $1<r_0\leq p_0<\infty$ and $1<r_1\leq p_1<\infty$, lead to \eqref{average:2} and we show in Theorem \ref{result:8} that, under suitable hypotheses, the regularity bound \eqref{average:2} is a consequence of the harmonic mean of $r_0$, $p_0$, $r_1$ and $p_1$ being equal to $2$, i.e.,
\begin{equation*}
	\left(\frac 14\left(\frac1{r_0}+\frac1{p_0}+\frac1{r_1}+\frac1{p_1}\right)\right)^{-1}=2.
\end{equation*}
For instance, provided $f(x,v)$ is compactly supported in $v$, we show in Corollaries \ref{result:9} and \ref{result:10} that \eqref{average:2} holds whenever:
\begin{itemize}
	\item $1<r_0\leq p_0'\leq 2\leq p_0<\frac{2n}{n-1}$, $r_1=p_0'$ and $p_1=r_0'$ (set $r_1=p_0'$ and $p_1=r_0'$ in Corollary~\ref{result:9}; in this case, further assume that $f$ is compactly supported in $x$ for technical reasons).
	\item $2\leq p_0<\frac{2n}{n-1}$, $r_0=p_0'$ and $r_1=p_1=2$ (set $r_0=p_0'$ and $r_1=p_1=2$ in either Corollary~\ref{result:9} or \ref{result:10}).
	\item $2\leq p_0\leq r_1'<\frac{2n}{n-1}$, $r_0=p_0$ and $p_1=r_1'$ (set $r_0=p_0$ and $p_1=r_1'$ in Corollary~\ref{result:10}).
\end{itemize}
It would certainly be interesting to determine the full optimal range of parameters $p_0$, $r_0$, $p_1$ and $r_1$ which lead to the smoothness of velocity averages \eqref{average:2}.

Overall, by providing a new duality principle in kinetic equations, the energy method greatly expands the knowledge of regularity theory in kinetic equations. It is to be emphasized, though, that the significance of precise formulations of averaging lemmas reaches beyond our purely academic interest in functional analysis. Indeed, it is widely believed that refined averaging lemmas will eventually lead to sharp regularity results in kinetic equations, which is of particular interest in models stemming from physical phenomena.

For example, observe that the energy method may be useful in the study of the regularity of renormalized solutions of the Boltzmann equation \eqref{boltzmann}. Indeed, renormalizations of particle densities $f(t,x,v)$ can easily be chosen so that they are uniformly bounded. However, renormalizing the Boltzmann collision operator $Q(f,f)$ only leads, in general, to a locally integrable right-hand side in \eqref{boltzmann} (or slightly better than integrable, thanks to entropy bounds). Is it possible to show that velocity averages of renormalizations of $f$ belong to $\dot H^{1/2}_{\rm loc}$? In view of \eqref{bounds:1}, this seems plausible, but remains unsettled for the moment.

Furthermore, velocity averaging lemmas have long been used to establish regularity properties of hyperbolic systems through the study of their kinetic formulations. For instance, Theorem 4 in \cite{lpt94:1} provides an example of an early application of averaging lemmas to kinetic formulations of scalar conservation laws, while Proposition 7 in \cite{lpt94:2} gives a similar result in the context of isentropic gas dynamics. A systematic and robust approach to establishing regularity in kinetic formulations via averaging lemmas is given in \cite{jp02}.

However, the regularity properties obtained by this procedure generally fail to match the expected optimal regularity in the corresponding non-linear hyperbolic systems. This shortcoming is likely a consequence of a lack of sharp suitable averaging lemmas (note that kinetic formulations often produce densities with very unique regularity and integrability properties in each variable, thereby requiring very specific averaging lemmas). Moreover, other powerful methods capable of reaching optimal regularity are sometimes available (see \cite{gp13}, for instance), which outperforms velocity averaging procedures. By establishing precise and sharp velocity averaging results, the hope is therefore to recover the optimal regularity for conservation laws. If successful, this procedure could lead to a robust method with potential applications to more kinetic formulations of other systems.

We acknowledge that answering these questions remains challenging, though. Nevertheless, the energy method contributes to this endeavor.

% =================
% = Energy method =
% =================
\section{The energy method}
\label{section:2}

We introduce here the energy method for the kinetic transport equation \eqref{transport:1} and deduce some new velocity averaging lemmas from it. These new results display a powerful duality principle between the particle density $f(x,v)$ and the source term $g(x,v)$ in \eqref{transport:1}, which, loosely speaking, allows us to trade integrability (and regularity) between $f(x,v)$ and $g(x,v)$.

\medskip
For any locally integrable function $f(x,v)$ defined on $\mathbb{R}^n\times\mathbb{R}^n$ and compactly supported in the $v$ variable, we define the velocity average of $f$ by
\begin{equation*}
	\tilde f(x)=\int_{\mathbb{R}^n}f(x,v)dv.
\end{equation*}
The Fourier transform of $f$ in all variables will be denoted by
\begin{equation*}
	\widehat f(\xi,\eta)=\iint_{\mathbb{R}^n\times\mathbb{R}^n}f(x,v)e^{-i(\xi\cdot x+\eta\cdot v)}dxdv,
\end{equation*}
so that the dual variable of $x$ (resp.\ $v$) is denoted by $\xi$ (resp.\ $\eta$). For convenience, we will also sometimes use the notation $\mathcal{F}f(\xi,\eta)=\widehat f(\xi,\eta)$. The inverse Fourier transform can then be denoted as
\begin{equation*}
	\mathcal{F}^{-1}h(x,v)=\frac1{(2\pi)^{2n}}\iint_{\mathbb{R}^n\times\mathbb{R}^n}h(\xi,\eta)e^{i(x\cdot\xi+v\cdot\eta)}d\xi d\eta.
\end{equation*}
Observe then that
\begin{equation*}
	\hspace{3pt}\widehat{\hspace{-3pt}\tilde f}(\xi)
	=\iint_{\mathbb{R}^n\times\mathbb{R}^n}f(x,v)e^{-i\xi\cdot x}dxdv
	=\widehat f(\xi,0).
\end{equation*}
It is to be emphasized that this simple formula has important consequences. In particular, it implies that, in order to evaluate the regularity of $\tilde f$, one only needs to control the decay of the trace of $\widehat f$ on $\{\eta=0\}$.

Next, recall that a tempered distribution $m(\xi,\eta)\in\mathcal{S}'(\mathbb{R}^n_\xi\times\mathbb{R}^n_\eta)$ defines a Fourier multiplier $m(D_x,D_v)$ on $L^p(\mathbb{R}^n_x;L^q(\mathbb{R}^n_v))$, for some $1\leq p,q\leq\infty$, if the mapping $m(D_x,D_v)$ given by
\begin{equation*}
	\mathcal F(m(D_x,D_v)f)=m(\xi,\eta)\widehat f(\xi,\eta),
\end{equation*}
for any $f\in \mathcal{S}(\mathbb{R}^n_x\times\mathbb{R}^n_v)$ ($\mathcal{S}$ denotes the Schwartz space of rapidly decreasing functions), can be extended into a bounded operator over $L^p(\mathbb{R}^n_x;L^q(\mathbb{R}^n_v))$. In particular, setting $m(\xi,\eta)=\xi$, we obtain, with the notation $D_{x}=-i\nabla_{x}$,
\begin{equation*}
	\mathcal{F}(D_xf)=\xi \widehat f(\xi,\eta)=\mathcal{F}(-i\nabla_xf).
\end{equation*}
We also have that
\begin{equation*}
	\mathcal{F}(vf(x,v))=i\nabla_\eta \widehat f(\xi,\eta)=-D_\eta\widehat f(\xi,\eta).
\end{equation*}
All in all, we have established the formula
\begin{equation}\label{fourier:1}
	\mathcal{F}(v\cdot\nabla_x f)(\xi,\eta)=-\xi\cdot \nabla_\eta\widehat f,
\end{equation}
so that the transport equation \eqref{transport:1} can be recast in Fourier variables as
\begin{equation*}
	\xi\cdot\nabla_\eta \widehat f(\xi,\eta) = - \widehat g(\xi,\eta).
\end{equation*}

Recall that we are looking to extract information on $\widehat f(\xi,0)$ from the preceding equation. This amounts to finding the value of $F=\widehat f$ on some (Lagrangian) manifold $\{\eta=0\}$ knowing that $F$ satisfies some transport equation. The energy method, based on the next lemma, is well-suited to this question and we expect it to yield a deeper insight into more intricate geometries, as well. 

Here, we only mention that, in the article \cite{MR1042388}, P.\ G\'erard develops a phase space interpretation of averaging lemmas via second microlocalization tools and that the paper \cite{MR1135922} by P.\ G\'erard \& F.\ Golse
provides a geometrical framework,
whereas the paper \cite{MR1659189} tackles the case of complex-valued vector-fields.
We hope to pursue this line of research in connection with the energy method in subsequent works.

\begin{lem}
Let $m(\xi,\eta)$ be a tempered distribution on $\mathbb R^n\times \mathbb R^n$. Then, one has the commutator identity
\begin{equation}\label{commutator:1}
	\bigl[ m(D_{x}, D_{v}), v\cdot \nabla_{x}\bigr]=\mu(D_{x},D_{v}),
	\quad \text{where }\mu(\xi,\eta)=\xi\cdot \frac{\partial m}{\partial \eta}.
\end{equation}
\end{lem}

\begin{proof}
Since the operator $v\cdot \nabla_{x}$ has the symbol $i v\cdot \xi$, which is a polynomial with degree $2$ whose Hessian has a null diagonal, the symbol of $\bigl[v\cdot \nabla_{x}, m(D_{x}, D_{v})\bigr]$ is exactly the Poisson bracket (see \cite[Corollary 1.1.22]{MR2599384})
$$
\frac1{i}\bigl\{m(\xi,\eta),i v\cdot \xi\bigr\}= \xi\cdot\frac{\partial m}{\partial \eta}(\xi,\eta),
$$
which is the sought result.
\end{proof}

\begin{rema}
	The preceding proof relies on the use of Poisson brackets, an important device with a geometrical content. For the convenience of readers who might not be familiar with such geometric tools, we provide here an alternative elementary justification of the preceding lemma, which should be more accessible. However, it should be noted that this approach is coordinate-dependent and, therefore, less robust.
	
	More precisely, for any $m(\xi,\eta)\in\mathcal{S}'(\mathbb{R}^n_\xi\times\mathbb{R}^n_\eta)$ and $f(x,v)\in \mathcal{S}(\mathbb{R}^n_x\times\mathbb{R}^n_v)$, we compute that
\begin{equation}\label{commutator:2}
	\begin{aligned}
		v\cdot\nabla_x\big[m(D_x,D_v)f(x,v)\big]
		&= \frac{1}{(2\pi)^{2n}}
		\iint_{\mathbb{R}^n\times\mathbb{R}^n}m(\xi,\eta)\widehat f(\xi,\eta)iv\cdot\xi e^{i(x\cdot\xi+v\cdot\eta)}d\xi d\eta
		\\
		&= \frac{1}{(2\pi)^{2n}}
		\iint_{\mathbb{R}^n\times\mathbb{R}^n}m(\xi,\eta)\widehat f(\xi,\eta)\xi\cdot\nabla_\eta e^{i(x\cdot\xi+v\cdot\eta)}d\xi d\eta
		\\
		&= \frac{1}{(2\pi)^{2n}}
		\iint_{\mathbb{R}^n\times\mathbb{R}^n}m(\xi,\eta)\xi\cdot \nabla_\eta\Big[\widehat f(\xi,\eta) e^{i(x\cdot\xi+v\cdot\eta)}\Big]d\xi d\eta
		\\
		&\quad - \frac{1}{(2\pi)^{2n}}
		\iint_{\mathbb{R}^n\times\mathbb{R}^n}m(\xi,\eta) \big[\xi\cdot \nabla_\eta\widehat f(\xi,\eta)\big] e^{i(x\cdot\xi+v\cdot\eta)}d\xi d\eta
		\\
		&= \frac{-1}{(2\pi)^{2n}}
		\iint_{\mathbb{R}^n\times\mathbb{R}^n}\big[\xi\cdot \nabla_\eta m(\xi,\eta)\big]\widehat f(\xi,\eta) e^{i(x\cdot\xi+v\cdot\eta)}d\xi d\eta
		\\
		&\quad +m(D_x,D_v)\big[v\cdot\nabla_xf(x,v)\big],
	\end{aligned}
\end{equation}
where we have used \eqref{fourier:1} in the last step.
This useful identity is a characterization of the commutator between the operators $v\cdot\nabla_x$ and $m(D_x,D_v)$.
It is only a reformulation of \eqref{commutator:1}.
\end{rema}

\begin{rema}
	In the case of functions $f(t,x,v)$, with $(t,x,v)\in\mathbb{R}\times\mathbb{R}^n\times\mathbb{R}^n$, depending on time $t$, it is readily seen that the commutator identity \eqref{commutator:1} can be adapted as
	\begin{equation*}
		\bigl[ m(D_t,D_{x}, D_{v}), \partial_t+v\cdot \nabla_{x}\bigr]=\mu(D_t,D_{x},D_{v}),
		\quad \text{where }\mu(\tau,\xi,\eta)=\xi\cdot \frac{\partial m}{\partial \eta}.
	\end{equation*}
	This identity provides then the basis to extend the results from this article to the time-dependent transport equation \eqref{transport:2}.
\end{rema}

From now on, we assume that the tempered distribution $m(\xi,\eta)$ is in fact a real-valued locally integrable function, which implies that $m(D_x,D_v)$ is a self-adjoint operator.
On the other hand, the vector field $v\cdot \nabla_{x}$ has null divergence and thus is skew-adjoint. As a result, employing the previous lemma, for any $f$ in the Schwartz space of $\mathbb R^{2n}$, we deduce the crucial energy identity
\begin{equation}\label{energy:1}
	\begin{aligned}
		2\operatorname{Re}\langle m(D_{x},D_{v}) f,v\cdot \nabla_{x}f\rangle_{L^2(\mathbb R^{2n})}
		&=\langle\bigl[m(D_{x},D_{v}) ,v\cdot \nabla_{x}\bigr] f, f\rangle_{L^2(\mathbb R^{2n})}
		\\
		&=\langle \mu(D_{x},D_{v}) f, f\rangle_{L^2(\mathbb R^{2n})}
		\\
		&=\iint_{\mathbb{R}^n\times\mathbb{R}^n} \xi\cdot \nabla_{\eta} m(\xi,\eta)\vert\widehat f(\xi,\eta)\vert^2 d\xi d\eta(2\pi)^{-2n}.
	\end{aligned}
\end{equation}
This property can be extended to any $f$ such that
$m(D_{x},D_{v}) f$ and $v\cdot \nabla_{x}f$ both belong to $L^2(\mathbb R^{2n})$.
Alternatively, the same identity can be obtained by multiplying \eqref{commutator:2} by $\overline{f(x,v)}$ and integrating in all variables.

The rule of the game now consists in finding a real-valued multiplier $m(\xi,\eta)$ which will allow us to deduce useful information on the velocity average $\tilde f(x)$, or equivalently on $\widehat f(\xi,0)$, from the energy identity \eqref{energy:1}.
\par
The following velocity averaging lemma is a simple and direct consequence of the energy identity for suitable choices of $m(\xi,\eta)$. It is the first main result of this work and is an extension of the classical velocity averaging lemmas from \cite{glps88}.

\begin{thm}\label{result:1}
	Let $f\in L^p(\mathbb{R}^n_x\times\mathbb{R}^n_v)$, with $p\in(1,\infty)$, be compactly supported in $v$ and such that $v\cdot\nabla_x f\in L^{p'}(\mathbb{R}^n_x\times\mathbb{R}^n_v)$. Then $\tilde f$ belongs to $\dot H^\frac 12(\mathbb{R}^n_x)$.
	
	More precisely, for any $p\in(1,\infty)$ and any compact set $K\subset\mathbb{R}^n_v$, there exists $C_{p,K}>0$ such that, for all $f(x,v)$ vanishing for $v$ outside $K$,
	\begin{equation*}
		\|\tilde f\|_{\dot H^\frac 12_x}^2\leq
		C_{p,K}\|f\|_{L^p_{x,v}}
		\|v\cdot\nabla_x f\|_{L^{p'}_{x,v}}.
	\end{equation*}
\end{thm}

\begin{proof}
	We provide in the appendix a few approximation lemmas allowing us to reduce the proofs of velocity averaging results to estimates for smooth compactly supported functions only. More specifically, here, we can use either Lemma \ref{approximation:3} or Lemma \ref{approximation:4} from the appendix to deduce that we only need to show the present theorem for smooth compactly supported functions. Thus, we assume now that $f(x,v)$ belongs to the Schwartz space of rapidly decreasing functions.
	
	\emph{We start with the case $n=1$.} The one-dimensional version of the energy identity \eqref{energy:1} reads
	\begin{equation}\label{energy:6}
		\iint_{\mathbb{R}\times\mathbb{R}}\xi \partial_\eta m_0(\xi,\eta)\big|\widehat f(\xi,\eta)\big|^2 d\xi d\eta
		=8\pi^2\operatorname{Re} \langle v\partial_xf,m_0(D_x,D_v)f\rangle_{L^2_{x,v}}.
	\end{equation}
	Then, choosing
	\begin{equation*}
		m_0(\xi,\eta)=\operatorname{sign}\xi\times\operatorname{sign}\eta
	\end{equation*}
	yields the identity
	\begin{equation}\label{energy:2}
		\|\tilde f\|_{\dot H^\frac 12(\mathbb{R})}^2=
		\frac 1{2\pi}\int_{\mathbb{R}}|\xi| \big|\widehat f(\xi,0)\big|^2 d\xi
		=2\pi\operatorname{Re} \langle v\partial_xf,m_0(D_x,D_v)f\rangle_{L^2_{x,v}}.
	\end{equation}
	
	Observe next that the operator $m_0(D_x,D_v)$ is a classical tensor product of Calder\'on--Zygmund singular integrals. More precisely, it is well-known that the Fourier multiplier operator produced by the function $\operatorname{sign}(\xi)$ is the Hilbert transform defined by
	\begin{equation}\label{hilbert}
		Hh(x)=\frac i\pi\operatorname{\mbox{p.v.}}\int_{\mathbb{R}}\frac 1{x-y}h(y)dy
		=\mathcal{F}^{-1}\big[(\operatorname{sign}\xi) \widehat h(\xi)\big](x),
	\end{equation}
	for any $h\in\mathcal{S}(\mathbb{R})$, where the singular integral is defined in the sense of Cauchy's principal value. We refer to \cite[Section 4.1]{g14} or \cite[Section 5.1]{hvvl16} for more details on the Hilbert transform and its properties. It follows that
	\begin{equation*}
		m_0(D_x,D_v)=H_x\otimes H_v,
	\end{equation*}
	where we used indices to distinguish the variable being targeted by the Hilbert transform.
	
	Now, since the Hilbert transform is bounded on all $L^p(\mathbb{R})$, for $1<p<\infty$, we deduce that $m_0(D_x,D_v)$ is bounded on $L^p(\mathbb{R}^2)$, as well. Therefore, we obtain from \eqref{energy:2} that
	\begin{equation*}
		\|\tilde f\|_{\dot H^\frac 12_x}^2
		=2\pi\operatorname{Re} \langle v\partial_xf,H_x\otimes H_vf\rangle_{L^2_{x,v}}
		\leq 2\pi C_p^2\|v\partial_xf\|_{L^{p'}_{x,v}}\|f\|_{L^{p}_{x,v}},
	\end{equation*}
	where $C_p>0$ is the operator norm of the Hilbert transform\footnote{It can be shown that $C_p=\tan\frac\pi{2p}$ for $1<p\leq 2$ and $C_p=\cot\frac\pi{2p}$ for $2\leq p<\infty$ (see Theorem 4.1.7 and the following remarks in \cite{g14}).} over $L^p(\mathbb{R})$, which concludes the one-dimensional proof.
	
	\emph{We handle now the case $n\geq 2$.} Replicating the one-dimensional strategy, we apply \eqref{energy:1} with
	\begin{equation*}
		m_1(\xi,\eta)=\operatorname{sign}\xi_1\times\operatorname{sign}\eta_1,
	\end{equation*}
	and we get
	\begin{equation}\label{energy:3}
		\iint_{(\xi,\eta')\in\mathbb{R}^n\times\mathbb{R}^{n-1}}|\xi_1|\big|\widehat f(\xi,0,\eta')\big|^2 d\xi d\eta'
		=(2\pi)^{2n}\operatorname{Re} \langle v\cdot\nabla_xf,m_1(D_x,D_v)f\rangle_{L^2_{x,v}},
	\end{equation}
	where $m_1(D_x,D_v)=H_{x_1}\otimes H_{v_1}$ is also a tensor product of Hilbert transforms, which is bounded over all $L^p(\mathbb{R}^n\times\mathbb{R}^n)$, for $1<p<\infty$.
	
	Next, further defining
	\begin{equation*}
		\tilde f_1(x,v')=\int_{\mathbb{R}}f(x,v_1,v')dv_1,
	\end{equation*}
	so that
	\begin{equation*}
		\widehat{\tilde f_1}(\xi,\eta')=
		\iint_{(x,v')\in\mathbb{R}^n\times\mathbb{R}^{n-1}}\left[\int_{\mathbb{R}}f(x,v_1,v')dv_1\right] e^{-i(\xi\cdot x+\eta'\cdot v')}dxdv'
		=\widehat f(\xi,0,\eta')
	\end{equation*}
	and
	\begin{equation*}
		|\xi_1|^\frac 12\widehat f(\xi,0,\eta')=|\xi_1|^\frac 12\widehat{\tilde f_1}(\xi,\eta')=\mathcal{F}\big(|D_{x_1}|^\frac 12 \tilde f_1\big)(\xi,\eta'),
	\end{equation*}
	we deduce from \eqref{energy:3} that
	\begin{equation}\label{energy:4}
		\begin{aligned}
			\big\||D_{x_1}|^\frac 12 \tilde f_1(x,v')\big\|_{L^2(\mathbb{R}^n\times\mathbb{R}^{n-1})}^2
			&=\frac 1{(2\pi)^{2n-1}}
			\iint_{\mathbb{R}^n\times\mathbb{R}^{n-1}}|\xi_1|\big|\widehat f(\xi,0,\eta')\big|^2 d\xi d\eta'
			\\
			&=2\pi\operatorname{Re} \langle v\cdot\nabla_xf,H_{x_1}\otimes H_{v_1}f\rangle_{L^2_{x,v}}
			\\
			&\leq 2\pi C_p^2\|v\cdot\nabla_xf\|_{L^{p'}_{x,v}}\|f\|_{L^{p}_{x,v}},
		\end{aligned}
	\end{equation}
	where $C_p>0$ is the operator norm of the Hilbert transform and only depends on $p$.
	
	Moreover, since the $v$-support of $f$ is compact and $|D_{x_1}|^\frac 12$ does not increase the support in the $v$-variable, we notice that
	\begin{equation}\label{support:1}
		\begin{aligned}
			\big\||D_{x_1}|^\frac 12\tilde f\big\|_{L^2_x}=
			\left\|\int_{\mathbb{R}^{n-1}}|D_{x_1}|^\frac 12\tilde f_1(x,v')dv'\right\|_{L^2_x}
			&\leq \int_{\mathbb{R}^{n-1}}\big\||D_{x_1}|^\frac 12\tilde f_1(x,v')\big\|_{L^2_x}dv'
			\\
			&\leq \sqrt{|\pi_{n-1}K|_{n-1}}\big\||D_{x_1}|^\frac 12 \tilde f_1(x,v')\big\|_{L^2_{x,v'}},
		\end{aligned}
	\end{equation}
	where $K$ is the compact set containing the $v$-support of $f$ and $\pi_{n-1}K$ stands for its orthogonal projection onto $\mathbb{R}^{n-1}_{v'}$.
	Thus, combining \eqref{energy:4} and \eqref{support:1}, we arrive at the estimate
	\begin{equation*}
		\big\||D_{x_1}|^\frac 12\tilde f\big\|_{L^2_x}^2
		\leq 2\pi C_p^2|\pi_{n-1}K|_{n-1}\|v\cdot\nabla_xf\|_{L^{p'}_{x,v}}\|f\|_{L^{p}_{x,v}}.
	\end{equation*}
	
	Finally, since there is nothing special about the first coordinate $x_1$, we may now repeat the preceding estimate for each remaining coordinate, which allows us to obtain
	\begin{equation*}
		\big\||D_{x}|^\frac 12\tilde f\big\|_{L^2_x}^2
		\leq
		\sum_{j=1}^n\big\||D_{x_j}|^\frac 12\tilde f\big\|_{L^2_x}^2
		\leq 2\pi C_p^2 \bigg(\sum_{j=1}^n|\pi_{n-1}^jK|_{n-1}\bigg)\|v\cdot\nabla_xf\|_{L^{p'}_{x,v}}\|f\|_{L^{p}_{x,v}},
	\end{equation*}
	where $\pi^j_{n-1}K$ now denotes the orthogonal projection of $K$ onto the coordinate hyperplane normal to the $v_j$-coordinate axis,
	thereby concluding the proof in any dimension.
\end{proof}

\begin{rema}
	The preprint \cite{jlt20} on arXiv
	displays a method for proving averaging lemmas which is very close to our approach
	and contains results which sometimes intersect with the contributions from this article. Specifically, Theorem \ref{result:1} above is the consequence of a basic implementation of the energy method. Since the ideas from \cite{jlt20} are based on a very similar approach (which the authors call the commutator method), it is not surprising to find a similar result therein. In fact, Theorem 1 from \cite{jlt20} (in the case $\varepsilon=0$) reaches the same conclusion as our Theorem \ref{result:1}. However, we note that Theorem 1 in \cite{jlt20} achieves a refined characterization of the velocity regularity of the densities, whereas our result only focuses on the regularity of velocity averages.
\end{rema}

\begin{rema}
	Summarizing the proof above, we note that we have used the energy identity \eqref{energy:1} in conjunction with the multiplier
	\begin{equation}\label{multiplier:1}
		M(\xi,\eta)=\sum_{1\leq j\leq n}\operatorname{sign}\xi_j\times\operatorname{sign}\eta_j,
	\end{equation}
	which is bounded on all $L^p(\mathbb{R}^{2n})$, with $1<p<\infty$. More precisely, denoting
	\begin{equation*}
		\tilde f_j(x,v_1,\ldots,v_{j-1},v_{j+1},\ldots,v_n)=\int_{\mathbb{R}}f(x,v)dv_j
	\end{equation*}
	and making repeated use of \eqref{support:1},
	the energy identity \eqref{energy:1} yields that
	\begin{equation}\label{energy:5}
		\begin{aligned}
			\big\||D_x|^\frac 12\tilde f\big\|_{L^2_x}^2
			&\leq
			\sum_{j=1}^n|\pi_{n-1}^jK|_{n-1}
			\big\||D_{x_j}|^\frac 12 \tilde f_j(x,v')\big\|_{L^2(\mathbb{R}^n\times\mathbb{R}^{n-1})}^2
			\\
			& \leq 2\pi \sup_{j=1,\ldots,n}|\pi_{n-1}^jK|_{n-1} \operatorname{Re} \langle v\cdot\nabla_xf,M(D_x,D_v)f\rangle_{L^2_{x,v}}.
		\end{aligned}
	\end{equation}
	It is then easy to conclude the proof employing the boundedness properties of $M(D_x,D_v)$.
\end{rema}

\begin{rema}
	Some comments about the endpoint cases of the preceding result are in order. As expected, the method of proof of Theorem \ref{result:1} fails for $p=1$ or $p=\infty$ due to the fact that singular integral operators (such as the Hilbert transform), and therefore Fourier multipliers, are generally not bounded over non-reflexive Lebesgue spaces. In fact, it would not be surprising to learn that the conclusion of Theorem \ref{result:1} fails to hold in these endpoint cases. A counterexample would certainly be interesting.
	
	It is in general possible, though, to obtain an endpoint result by substituting the Lebesgue spaces $L^1_{x,v}$ and $L^\infty_{x,v}$ by Hardy spaces $\mathfrak{H}^1$ and spaces of functions of bounded mean oscillation $\mathrm{BMO}$, respectively (we refer to \cite{g14:2} for a complete introduction to these spaces). However, some technical care is required.
	
	Indeed, the Hardy space $\mathfrak H^1(\mathbb{R}^n)$ is generally defined as the subspace of functions $f(x)$ in $L^1(\mathbb{R}^n)$ such that all Riesz transforms $(D_j/|D|)f(x)$, with $j=1,\ldots,n$, are integrable, as well. A fundamental feature of Hardy spaces is the fact that, even though Calder\'on--Zygmund singular integral operators are not bounded over $L^1(\mathbb{R}^n)$, they are bounded on $\mathfrak H^1(\mathbb{R}^n)$ (see \cite[Section 6.7.1]{g14:2}).\footnote{Standard pseudo-differential operators (of order $0$) are also bounded on $L^p$, with $1<p<\infty$, but not on $L^1$ (see \cite{t91} for some $L^p$ estimates). For the $\mathfrak H^1$ boundedness, some additional condition---yet to be unraveled---should be satisfied by the symbol. It seems possible to write down a decent sufficient condition. However, we must of course keep in mind that $\mathfrak H^1$ is not a local space and that, for instance, the integral of each element of $\mathfrak H^1$ must be $0$, which is a property that is in general ruined by the multiplication by a cutoff function.}
	Moreover, whereas $L^1(\mathbb{R}^n)$ is dual to no space, the Hardy space $\mathfrak H^1(\mathbb{R}^n)$ is dual to the space of functions with vanishing mean oscillation $\mathrm{VMO}(\mathbb{R}^n)$ (defined as the closure of continuous functions that vanish at infinity $C_0(\mathbb{R}^n)$ in the $\mathrm{BMO}(\mathbb{R}^n)$ norm, i.e., the dual norm on $\mathfrak H^1(\mathbb{R}^n)$; see \cite{g02} for a survey of related results). These remarkable properties of $\mathfrak H^1$ make this functional space more suitable than $L^1$ in many problems related to harmonic analysis. In particular, they have deep consequences in the analysis of partial differential equations, as illustrated, for instance, by the results from \cite{as11}, \cite{MR1225511}, \cite{em94} and \cite{m90}.
	
	Unfortunately, even though tensor products of classical singular integral operators are bounded over $L^p(\mathbb{R}^n)$, with $1<p<\infty$, they do not enjoy the same boundedness properties over $\mathfrak H^1(\mathbb{R}^n)$. Indeed, take a function $w(\xi,\eta)$ in the Schwartz space of $\mathbb{R}^2$ such that $w(0,0)=0$. Then, the Fourier transform $f(x,v)=\widehat w(x,v)$ belongs to $\mathfrak H^1\cap \mathcal{S}(\R^2)$. However, the coodinatewise Hilbert transform $H_x f(x,v)$ is not integrable (one may also consider $H_vf(x,v)$ or $H_x\otimes H_vf(x,v)$ for a similar argument), otherwise this would imply that $(\sign\xi) w(\xi,\eta)$ is continuous, since $\mathcal{F}L^1\subset C_0$ by the Riemann--Lebesgue lemma. This continuity would then force $w(0,\eta)=0$, for all $\eta\in\mathbb{R}$, which does not hold in general.\footnote{For pseudo-differential operators, tensor products are often poorly behaved: Taking two symbols $a(x,\xi)$ and $b(v,\eta)$ of order $0$, the first derivative of the product $a(x,\xi)b(v,\eta)$ will in general not decay as $(1+\val \xi+\val\eta)^{-1}$,
but will decay like $(1+\val \xi)^{-1}+(1+\val \eta)^{-1}$.}
	
	In conclusion, in order to deal with the boundedness of the Fourier multiplier $M(D_x,D_v)$ defined in \eqref{multiplier:1} over Hardy-type spaces and thereby extend Theorem \ref{result:1} to the endpoint settings $p=1$ and $p=\infty$, one is compelled to resort to the use of tools from harmonic analysis on product spaces (see \cite{f87} for such tools relevant to the double Hilbert transform). These methods have previously been used successfully in the context of velocity averaging lemmas by B\'ezard \cite{b94}\footnote{B\'ezard used a product Hardy space in $x$ which depends on the velocity variable $v$. More precisely, given any $v\in\mathbb{R}^n\setminus\{0\}$, he introduced the Hardy space $\mathfrak H^1\big(\mathbb{R}_v\times\mathbb{R}^{n-1}_{v^\perp}\big)$ based on the product structure $\mathbb{R}^n=\mathbb{R}v\oplus(\mathbb{R}v)^\perp$.} and it seems plausible that a similar approach based on product Hardy spaces applies in the context of Theorem \ref{result:1}. However, this is rather technical and, therefore, for the sake of simplicity, we will not be going into further detail on this topic.
\end{rema}

\begin{rema}
	The results presented here are global: global assumptions, global results. We want now to address the question of localization of Theorem \ref{result:1}.
	\par
	Let $\Omega$ be an open subset of $\mathbb{R}^n_x\times\mathbb{R}^n_v$. For all $\phi\in C_c^\infty(\Omega)$, we define
	\begin{equation*}
		\tilde f_\phi(x)=\int_{\mathbb{R}^n}f(x,v)\phi(x,v)dv.
	\end{equation*}
	Then, for any $f\in L^p_{\rm loc}(\Omega)$ and $v\cdot\nabla_x f\in L^{p'}_{\rm loc}(\Omega)$, with $2\leq p<\infty$, we have
	\begin{equation*}
		v\cdot\nabla_x(\phi f)=
		\underbrace{(v\cdot\nabla_x\phi)f}_{\in L^p_{\rm comp}\subset L^{p'}_{\rm comp}}
		+\underbrace{\phi v\cdot\nabla_x f}_{\in L^{p'}_{\rm comp}},
	\end{equation*}
	so that $v\cdot\nabla_x(\phi f)$ belongs to $L^{p'}$ and $\phi f\in L^p$. It therefore follows from Theorem \ref{result:1} that $\tilde f_\phi=\widetilde{\phi f}\in H^\frac 12$.
	
	Observe that a similar localization fails in the case $1<p<2$. Somehow, the duality principle stemming from the energy method cannot be localized when $1<p<2$. Of course, one may also raise the question of \emph{microlocalization} of Theorem \ref{result:1}, which would lead to the same type of constraint.
\end{rema}

Our next result is a simple generalization of Theorem \ref{result:1} to mixed Lebesgue spaces. Recall that the mixed Lebesgue spaces $L^p(\mathbb{R}^n_x;L^q(\mathbb{R}^n_v))$, with $1\leq p,q\leq \infty$, are defined as the Banach spaces of measurable functions $f(x,v)$ over $\mathbb{R}^n_x\times\mathbb{R}^n_v$ endowed with the norm
\begin{equation*}
	\|f\|_{L^p(\mathbb{R}^n_x;L^q(\mathbb{R}^n_v))}
	=\big\|\|f\|_{L^q(\mathbb{R}^n_v)}\big\|_{L^p(\mathbb{R}^n_x)}.
\end{equation*}
We refer to \cite{bp61} for a systematic treatment of mixed Lebesgue spaces including all their essential properties. In particular, when $1\leq p,q<\infty$, one has that smooth compactly supported functions are dense in $L^p(\mathbb{R}^n_x;L^q(\mathbb{R}^n_v))$ and it is shown in \cite{bp61} that dual spaces are given by
\begin{equation*}
	\big(L^p(\mathbb{R}^n_x;L^q(\mathbb{R}^n_v))\big)'
	=
	L^{p'}(\mathbb{R}^n_x;L^{q'}(\mathbb{R}^n_v)),
\end{equation*}
for the range of parameters $1\leq p,q<\infty$, with the usual identification of linear forms in the $L^2_{x,v}$ inner product.

Furthermore, whenever $1\leq p,q<\infty$, the mixed Lebesgues spaces $L^p(\mathbb{R}^n_x;L^q(\mathbb{R}^n_v))$ can also be interpreted as Banach-space-valued Lebesgue spaces (or Bochner spaces), i.e., spaces of $L^p(\mathbb{R}^n_x;X)$ functions taking values in $X=L^q(\mathbb{R}^n_v)$. We refer the reader to \cite[Chapter 1]{hvvl16} for a complete introduction to Bochner spaces (see Proposition 1.2.25 therein for the identification of vector-valued functions with functions of two variables).

This interpretation is particularly important when dealing with the boundedness of Fourier multipliers and singular integral operators on vector-valued functions.
More precisely, we will make repeated use of the fact that the Hilbert transform \eqref{hilbert} defines an operator on vector-valued functions $h(x)\in C^1_c(\mathbb{R};X)$ that extends into a bounded operator over $L^p(\mathbb{R};X)$, for any given $1<p<\infty$, if and only if $X$ is a Banach space with the property of unconditional martingale differences (the UMD property; see \cite[Section 5.1]{hvvl16}). Fortunately, it can be shown that reflexive mixed Lebesgue spaces (i.e., mixed Lebesgue spaces with all integrability parameters contained strictly between $1$ and $\infty$) have the UMD property (see \cite[Proposition 4.2.15]{hvvl16}).

All in all, observe that the boundedness properties of vector-valued Hilbert transforms allow us to deduce that the Fourier multiplier operator \eqref{multiplier:1} is bounded over all mixed Lebesgue spaces $L^p(\mathbb{R}^n_x;L^q(\mathbb{R}^n_v))$ and $L^p(\mathbb{R}^n_v;L^q(\mathbb{R}^n_x))$, with $1< p,q< \infty$. This crucial observation leads to an extension of the duality principle from Theorem \ref{result:1} to mixed Lebesgue spaces, which is the content of the next result.

\begin{thm}\label{result:2}
	Let $f\in L^p(\mathbb{R}^n_x;L^q(\mathbb{R}^n_v))$, with $p,q\in(1,\infty)$, be compactly supported (in all variables) and such that $v\cdot\nabla_x f\in L^{p'}(\mathbb{R}^n_x;L^{q'}(\mathbb{R}^n_v))$. Then $\tilde f$ belongs to $\dot H^\frac 12(\mathbb{R}^n_x)$.
	
	More precisely, for any $p,q\in(1,\infty)$ and any compact set $K\subset\mathbb{R}^n_v$, there exists $C_{p,q,K}>0$ such that, for all compactly supported $f(x,v)$ vanishing for $v$ outside $K$,
	\begin{equation*}
		\|\tilde f\|_{\dot H^\frac 12_x}^2\leq
		C_{p,q,K}\|f\|_{L^p_{x}L^q_v}
		\|v\cdot\nabla_x f\|_{L^{p'}_{x}L^{q'}_v}.
	\end{equation*}
	In particular, the constant $C_{p,q,K}$ is independent of the size of the support of $f(x,v)$ in $x$. Moreover, whenever $\frac qp>\frac{n-1}n$, the above inequality remains true for functions that are compactly supported in $v$ but not necessarily in $x$.
\end{thm}

\begin{proof}
	Thanks to Lemma \ref{approximation:5} from the appendix, observe that we only need to consider smooth compactly supported functions. Furthermore, if $\frac qp>\frac{n-1}n$, it is readily seen that both constraints \eqref{constraint:1} and \eqref{constraint:2} hold, which implies that the theorem extends to functions that are not necessarily compactly supported in $x$ by Lemmas \ref{approximation:3} and \ref{approximation:4}. Either way, we only need now to consider functions $f(x,v)\in C_c^\infty(\mathbb{R}^n\times\mathbb{R}^n)$.
	
	We use the method of proof of Theorem \ref{result:1}. Thus, we observe that the energy identity \eqref{energy:5} yields that
	\begin{equation*}
		\begin{aligned}
			\big\||D_x|^\frac 12\tilde f\big\|_{L^2_x}^2
			& \leq
			2\pi \sup_{j=1,\ldots,n}|\pi_{n-1}^jK|_{n-1}\operatorname{Re} \langle v\cdot\nabla_xf,M(D_x,D_v)f\rangle_{L^2_{x,v}}
			\\
			&\leq 2\pi \sup_{j=1,\ldots,n}|\pi_{n-1}^jK|_{n-1}\|v\cdot\nabla_xf\|_{L^{p'}_xL^{q'}_v}\|M(D_x,D_v)f\|_{L^{p}_xL^q_v}.
		\end{aligned}
	\end{equation*}
	Then, further employing the previously discussed boundedness properties of the singular integral operator $M(D_x,D_v)$ over mixed Lebesgue spaces, we find that
	\begin{equation*}
		\big\||D_x|^\frac 12\tilde f\big\|_{L^2_x}^2
		\leq C_{p,q,K}\|v\cdot\nabla_xf\|_{L^{p'}_xL^{q'}_v}\|f\|_{L^{p}_xL^q_v},
	\end{equation*}
	for some $C_{p,q,K}>0$, which concludes the proof.
\end{proof}

\begin{rema}
	As before, we note that the conclusions of Theorem \ref{result:2} above intersect with the results from \cite{jlt20}. To be precise, Corollary 1 from \cite{jlt20} (in the case $\varepsilon=0$) provides a gain of almost half a derivative when the density and the source term lie in mixed Lebesgue spaces which are dual to each other, with some limitations on the integrability parameters. In contrast, our Theorem \ref{result:2} achieves a gain of exactly half a derivative in a similar setting and does not require any bounds on the integrability parameters.
\end{rema}

The reader will certainly have gathered by now that the duality principle embodied by Theorems \ref{result:1} and \ref{result:2} can be generalized to any suitable couple of dual functional spaces. For example, the velocity averaging result from Theorem \ref{result:2} also holds if one replaces the mixed Lebesgue spaces
\begin{equation*}
	L^p(\mathbb{R}^n_x;L^q(\mathbb{R}^n_v))
	\quad\text{and}\quad
	L^{p'}(\mathbb{R}^n_x;L^{q'}(\mathbb{R}^n_v))
\end{equation*}
by, respectively,
\begin{equation*}
	L^p(\mathbb{R}^n_v;L^q(\mathbb{R}^n_x))
	\quad\text{and}\quad
	L^{p'}(\mathbb{R}^n_v;L^{q'}(\mathbb{R}^n_x)).
\end{equation*}
The proof follows \emph{mutatis mutandis.}

Moreover, using the Bessel potentials
\begin{equation*}
	\begin{aligned}
		(1-\Delta_x)^\frac a2 & =(1+|D_x|^2)^\frac a2,
		\\
		(1-\Delta_v)^\frac \alpha 2 & =(1+|D_v|^2)^\frac \alpha 2,
	\end{aligned}
\end{equation*}
with $a,\alpha\in\mathbb{R}$, it is also possible to vary the required regularity of $f$ and $v\cdot\nabla_x f$ in the preceding results. Indeed, it is readily seen from the proof of Theorem \ref{result:2}, using that the Bessel potentials are self-adjoint and commute with $M(D_x,D_x)$, that $\tilde f$ belongs to $\dot H^\frac 12_x(\mathbb{R}^n_x)$ as soon as
\begin{equation}\label{duality:1}
	\begin{aligned}
		(1-\Delta_x)^\frac a2 (1-\Delta_v)^\frac \alpha 2 f & \in L^p(\mathbb{R}^n_x;L^q(\mathbb{R}^n_v))
		\quad\text{(resp.\ $L^p(\mathbb{R}^n_v;L^q(\mathbb{R}^n_x))$),}
		\\
		(1-\Delta_x)^{-\frac a2} (1-\Delta_v)^{-\frac \alpha 2}(v\cdot\nabla_x f) & \in L^{p'}(\mathbb{R}^n_x;L^{q'}(\mathbb{R}^n_v))
		\quad\text{(resp.\ $L^{p'}(\mathbb{R}^n_v;L^{q'}(\mathbb{R}^n_x))$).}
	\end{aligned}
\end{equation}
We will come back to regularity issues in the energy method in the next section. The variants on this duality principle for velocity averaging lemmas are now pretty much endless (e.g., introducing Besov spaces) and we will refrain from delving any further into more examples. At this stage, the method is arguably more important than the results.

\medskip

We conclude this section by observing that most previously existing velocity averaging results relied on the construction of a parametrix, i.e., an approximate inverse representation formula, for the kinetic transport equations \eqref{transport:2} or \eqref{transport:1}. These representation formulas have proven to be a powerful tool in kinetic theory and have already largely contributed to the discovery of an abundance of velocity averaging results (e.g., see \cite{am14}, \cite{am19}, \cite{b94}, \cite{dp01},\footnote{
Theorem 1.2 from \cite{dp01} is based on a subtle interpolation procedure and provides a sharp velocity averaging lemma in Besov spaces.
Although it is quite likely that a duality energy method can be used in that case, the detailed proofs remain to be written.
} \cite{dlm91}, \cite{gs02}, \cite{jv04}, \cite{w02}). However, the duality principle provided by the energy method seems to be missing from the insight brought by such parametrix representations (except in the case $n=1$ where the duality principle can be recovered by interpolation methods from the estimates of \cite[Section 4]{am19}). In fact, it remains unclear whether such formulas are even able to capture the duality principle at all, when $n\geq 2$.

We will nevertheless come back to the use of parametrix representation formulas for \eqref{transport:1} in conjunction with the energy method in Section \ref{section:dispersion} below, where we explore the effects of dispersion on the energy method.

% ====================
% = Singular sources =
% ====================
\section{Singular sources and hypoellipticity}
\label{section:3}

It is not uncommon to encounter situations where the source term $g(x,v)$ in \eqref{transport:1} (or in \eqref{transport:2}) is singular. In the history of velocity averaging, such singular settings emerged naturally in the study of the Vlasov--Maxwell system \eqref{VM} by considering the Vlasov term $(E+v\times B)\cdot\nabla_vf$ as a source term in the kinetic transport equation. This led the authors of \cite{dl89:2} (see Section 3 therein) to consider the case where the source $g(x,v)$ in \eqref{transport:1} takes the form
\begin{equation*}
	g=(1-\Delta_v)^\frac\beta 2h
	\text{ with }
	h\in L^2(\mathbb{R}^n_x\times\mathbb{R}^n_v),
\end{equation*}
for some $\beta\geq 0$, which results in a decreased regularity on velocity averages.

Another fundamental example of a similar singular setting can be found in the Boltzmann equation \eqref{boltzmann} for long-range interactions (as well as the Landau equation). Indeed, in this case, the collision operator $Q(f,f)$ is known to act as a non-linear fractional diffusion in the velocity variable. For more details, we refer to \cite{av02}, where a notion of global renormalized solutions was constructed for the Boltzmann equation without cutoff (see also \cite{av04} concerning the Landau equation).

We also mention here that kinetic formulations of non-linear systems typically lead to kinetic transport equations with singular source terms. For example, the regularity results from \cite{jp02} are based on applications of velocity averaging lemmas with singular source terms to kinetic formulations of multidimensional scalar conservation laws, one-dimensional isentropic gas dynamics and Ginzburg--Landau models in micromagnetics.

Due to the significance of such singular settings, we are now going to momentarily stray away from the topic of maximal velocity averaging (in the sense that we will investigate cases where $\tilde f$ might not belong to $\dot H^{1/2}_x$) and explore the impact of singular source terms on the energy method. Inevitably, the important hypoelliptic phenomenon in kinetic equations established in \cite{b02} (see also \cite{as11}), which transfers regularity from the velocity variable $v$ to the spatial variable $x$, will also play a role in the results of this section.

Many versions of velocity averaging in the classical Hilbertian setting $L^2_{x,v}$ (including hypoellipticity and singular sources) can now be found in the literature. Nevertheless, for a convenient reference, we point to \cite{am14}, where perhaps the most comprehensive version of Hilbertian velocity averaging can be found in Theorem 4.5 therein. Thus, we will be comparing the regularity of velocity averages obtained in the results of this section with the regularity established in Theorem~4.5 from \cite{am14}. Quite satisfyingly, we will see that the respective gains of regularity match.

Our next result shows the versatility of the energy method by establishing the regularization of velocity averages in the presence of singular sources, taking into account the hypoelliptic transfer of regularity from $v$ to $x$.

\begin{thm}\label{result:3}
	Let the family $\{f_\lambda(x,v)\}_{\lambda\in\Lambda}\subset\mathcal{S}(\mathbb{R}^n\times\mathbb{R}^n)$ be uniformly compactly supported in $v$ and such that the following subsets are bounded
	\begin{equation*}
		\begin{aligned}
			\big\{(1-\Delta_x)^\frac a 2(1-\Delta_v)^\frac \alpha 2f_\lambda\big\}_{\lambda\in\Lambda} & \subset L^p(\mathbb{R}^n_x;L^q(\mathbb{R}^n_v)),
			\\
			\big\{(1-\Delta_x)^\frac b 2(1-\Delta_v)^\frac \beta 2v\cdot\nabla_x f_\lambda\big\}_{\lambda\in\Lambda} & \subset L^{p'}(\mathbb{R}^n_x;L^{q'}(\mathbb{R}^n_v)),
			\\
			\big\{(1-\Delta_x)^\frac c 2(1-\Delta_v)^\frac \gamma 2f_\lambda\big\}_{\lambda\in\Lambda} & \subset L^2(\mathbb{R}^n_x\times\mathbb{R}^n_v),
		\end{aligned}
	\end{equation*}
	for some given $1<p,q<\infty$ and any regularity parameters $a,b,c,\alpha,\beta,\gamma\in\mathbb{R}$ satisfying the constraints
	\begin{equation*}
		1+a+b-2c\geq 0,
		\quad
		\alpha+\beta\leq 0,
		\quad
		\gamma>-\frac 12.
	\end{equation*}
	Then $\{\tilde f_\lambda\}_{\lambda\in\Lambda}$ is a bounded family in $H^\sigma(\mathbb{R}^n_x)$, where
	\begin{equation*}
		\sigma=
		\frac{1+a+b-2c}{1-\alpha-\beta+2\gamma}\Big(\frac 12+\gamma\Big)+c.
	\end{equation*}
	
	More precisely, for any compact set $K\subset\mathbb{R}^n_v$, there exists $C_{K}>0$ (which also depends on all integrability and regularity parameters) such that, for all $f(x,v)\in\mathcal{S}(\mathbb{R}^n\times\mathbb{R}^n)$ vanishing for $v$ outside $K$,
	\begin{equation*}
		\begin{aligned}
			\|\tilde f\|_{H^\sigma_x}^2
			&\leq
			C_{K}\|(1-\Delta_x)^\frac a 2(1-\Delta_v)^\frac \alpha 2f\|_{L^p_x L^q_v}
			\|(1-\Delta_x)^\frac b 2(1-\Delta_v)^\frac \beta 2v\cdot\nabla_x f\|_{L^{p'}_x L^{q'}_v}
			\\
			&\quad +C_K
			\|(1-\Delta_x)^\frac c 2(1-\Delta_v)^\frac \gamma 2f\|_{L^2_{x,v}}^2.
		\end{aligned}
	\end{equation*}
\end{thm}

\begin{rema}
	Using Lemma \ref{approximation:6} from the appendix, it is possible to extend the validity of the above inequality to any compactly supported function $f(x,v)$ vanishing for $v$ outside $K$ and satisfying the bounds
	\begin{equation*}
		\begin{aligned}
			(1-\Delta_x)^\frac a 2(1-\Delta_v)^\frac \alpha 2f & \in L^p(\mathbb{R}^n_x;L^q(\mathbb{R}^n_v)),
			\\
			(1-\Delta_x)^\frac b 2(1-\Delta_v)^\frac \beta 2v\cdot\nabla_x f & \in L^{p'}(\mathbb{R}^n_x;L^{q'}(\mathbb{R}^n_v)),
			\\
			(1-\Delta_x)^\frac c 2(1-\Delta_v)^\frac \gamma 2f & \in L^2(\mathbb{R}^n_x\times\mathbb{R}^n_v),
			\\
			(1-\Delta_x)^\frac b 2(1-\Delta_v)^\frac \beta 2 f & \in L^{p'}(\mathbb{R}^n_x;L^{q'}(\mathbb{R}^n_v)),
		\end{aligned}
	\end{equation*}
	where the constant $C_K$ remains independent of the size of the support of $f(x,v)$ in $x$ and the norm of $(1-\Delta_x)^\frac b 2(1-\Delta_v)^\frac \beta 2 f$ in $L^{p'}(\mathbb{R}^n_x;L^{q'}(\mathbb{R}^n_v))$.
\end{rema}

\begin{rema}
	In a first reading of the statement of Theorem \ref{result:3}, it may not seem optimal to impose two distinct bounds
	\begin{equation*}
		\big\{(1-\Delta_x)^\frac a 2(1-\Delta_v)^\frac \alpha 2f_\lambda\big\}_{\lambda\in\Lambda} \subset L^p(\mathbb{R}^n_x;L^q(\mathbb{R}^n_v))
	\end{equation*}
	and 
	\begin{equation*}
		\big\{(1-\Delta_x)^\frac c 2(1-\Delta_v)^\frac \gamma 2f_\lambda\big\}_{\lambda\in\Lambda} \subset L^2(\mathbb{R}^n_x\times\mathbb{R}^n_v)
	\end{equation*}
	on the family of functions $\{f_\lambda\}_{\lambda\in\Lambda}$.
	Nevertheless, these bounds are a reflection of the fact that velocity averaging lemmas often display different patterns of regularization depending on whether the integrability parameters belong to the range $(1,2]$ or the range $[2,\infty)$.
	Thus, by imposing bounds in $L^p_xL^q_v$ as well as in $L^2_{x,v}$, we are able to formulate a unified statement of Theorem \ref{result:3} that applies across all parameters $1<p,q<\infty$.
	Later on, following the proof of Theorem \ref{result:3}, we will provide direct corollaries where the need for distinct bounds in $L^p_xL^q_v$ and $L^2_{x,v}$ is removed. These corollaries certainly seem more practical and they can be reconciled with the results from the previous section, but their applicability requires us to discriminate the relative size of the integrability parameters with respect to the value $2$.
\end{rema}

\begin{proof}
	We are going to utilize the rather common notation
	\begin{equation*}
		\langle r\rangle=(1+|r|^2)^\frac 12,
	\end{equation*}
	for any $r\in\mathbb{R}^n$, so that $(1-\Delta_x)^\frac 12=\langle D_x\rangle$ and $(1-\Delta_v)^\frac 12=\langle D_v\rangle$.
	Moreover, we will sometimes write $A\lesssim B$ to denote that $A\leq CB$, for some constant $C>0$ that only depends on fixed parameters.
	
	\emph{We consider the case $n=1$, first.} We start by introducing a cutoff function $\chi(r)\in C_c^\infty(\mathbb{R})$ with the property that
	\begin{equation}\label{cutoff:2}
		\mathds{1}_{\{|r|\leq\frac 12\}}\leq\chi(r)\leq\mathds{1}_{\{|r|\leq 1\}},
	\end{equation}
	for every $r\in\mathbb{R}$. We then employ the one-dimensional energy identity \eqref{energy:6} with the multiplier
	\begin{equation*}
		m_0(\xi,\eta)=\frac{\operatorname{sign}\xi\times \operatorname{sign}\eta}{\langle \xi\rangle^{1-2\sigma}}
		\chi\Big(\frac{\eta}{\langle \xi\rangle^s}\Big),
	\end{equation*}
	where the parameter $s\geq 0$ is to be determined later on.
	In particular, we compute that
	\begin{equation*}
		\xi \partial_\eta m_0(\xi,\eta)= \frac{2|\xi|}{\langle\xi\rangle^{1-2\sigma}}\delta_0(\eta)
		+\frac{|\xi|\operatorname{sign}\eta}{\langle\xi\rangle^{1+s-2\sigma}} \chi'\Big(\frac{\eta}{\langle\xi\rangle^s}\Big)
	\end{equation*}
	which, when incorporated into \eqref{energy:6}, leads to
	\begin{equation*}
		\begin{aligned}
			2\int_{\mathbb{R}}\frac{|\xi|}{\langle\xi\rangle}\langle\xi\rangle^{2\sigma}\big|\widehat f(\xi,0)\big|^2 d\xi
			& =8\pi^2\operatorname{Re} \langle v\partial_xf,m_0(D_x,D_v)f\rangle_{L^2_{x,v}}
			\\
			&\quad -\iint_{\mathbb{R}\times\mathbb{R}}\frac{|\xi|\operatorname{sign}\eta}{\langle\xi\rangle^{1+s-2\sigma}} \chi'\Big(\frac{\eta}{\langle\xi\rangle^s}\Big)\big|\widehat f(\xi,\eta)\big|^2 d\xi d\eta.
		\end{aligned}
	\end{equation*}
	The use of the truncation $\chi\big(\frac{\eta}{\langle\xi\rangle^s}\big)$ is typical when dealing with the hypoelliptic phenomenon and the transfer of regularity between $x$ and $v$ in kinetic equations. For example, the reader will find similar truncations being employed in \cite[Section 4.4]{as11}.
	
	Next, observe that $\chi'\big(\frac{\eta}{\langle\xi\rangle^s}\big)$ is supported on
	\begin{equation*}
		\Big\{\frac 12 \langle\xi\rangle^s\leq |\eta|\leq \langle\xi\rangle^s\Big\}
		\subset
		\Big\{\frac 12 \langle\xi\rangle^s\leq \langle\eta\rangle\leq \sqrt 2 \langle\xi\rangle^s\Big\}
		\subset
		\Big\{1\leq 4^{|\gamma|}\frac{\langle\eta\rangle^{2\gamma}}{\langle\xi\rangle^{2\gamma s}}\Big\}
	\end{equation*}
	so that we may deduce
	\begin{equation*}
		\begin{aligned}
			\int_{\mathbb{R}}\frac{|\xi|}{\langle\xi\rangle}\langle\xi\rangle^{2\sigma}
			\big|\hspace{3pt}\widehat{\hspace{-3pt}\tilde f}(\xi)\big|^2 d\xi
			& \leq 4\pi^2\operatorname{Re}
			\Big\langle
				\langle D_x\rangle^{b}\langle D_v\rangle^{\beta}
				v\partial_xf,\frac{m_0(D_x,D_v)}{\langle D_x\rangle^{a+b}\langle D_v\rangle^{\alpha+\beta}}
				\langle D_x\rangle^{a}\langle D_v\rangle^{\alpha} f
			\Big\rangle_{L^2_{x,v}}
			\\
			&\quad +\frac {4^{|\gamma|}\|\chi'\|_{L^\infty}}2
			\iint_{\mathbb{R}\times\mathbb{R}}
				\frac{|\xi|\langle \xi\rangle^{2c}\langle \eta\rangle^{2\gamma}}
				{\langle \xi\rangle^{1+s(1+2\gamma)+2c-2\sigma}}
				\big|\widehat f(\xi,\eta)\big|^2
			d\xi d\eta.
		\end{aligned}
	\end{equation*}
	The denominator in the last integrand above dictates the appropriate choice for the parameter $s\geq 0$. More precisely, we set
	\begin{equation}\label{parameter:1}
		s=\frac{2(\sigma -c)}{1+2\gamma}=\frac{1+a+b-2c}{1-\alpha-\beta+2\gamma},
	\end{equation}
	whence
	\begin{equation}\label{estimate:1}
		\begin{aligned}
			\int_{\mathbb{R}}\frac{|\xi|}{\langle\xi\rangle}\langle\xi\rangle^{2\sigma}
			\big|\hspace{3pt}\widehat{\hspace{-3pt}\tilde f}(\xi)\big|^2 d\xi
			& \leq 4\pi^2
			\left\|
				\langle D_x\rangle^{b}\langle D_v\rangle^{\beta}v\partial_xf
			\right\|_{L^{p'}_x L^{q'}_v}
			\left\|
				\frac{m_0(D_x,D_v)}{\langle D_x\rangle^{a+b}\langle D_v\rangle^{\alpha+\beta}}
				\langle D_x\rangle^{a}\langle D_v\rangle^{\alpha} f
			\right\|_{L^p_x L^q_v}
			\\
			&\quad +2\pi^2 4^{|\gamma|}\|\chi'\|_{L^\infty}
			\left\|
				\langle D_x\rangle^{c}\langle D_v\rangle^{\gamma} f
			\right\|_{L^2_{x,v}}^2.
		\end{aligned}
	\end{equation}
	We address now the somewhat subtle step of establishing the boundedness of the multiplier
	\begin{equation*}
		\frac{m_0(\xi,\eta)}{\langle \xi\rangle^{a+b}\langle \eta\rangle^{\alpha+\beta}}
		=\frac{\operatorname{sign}\xi \times \operatorname{sign}\eta}{\langle \xi\rangle^{1+a+b-2\sigma}\langle \eta\rangle^{\alpha+\beta}}
		\chi\Big(\frac{\eta}{\langle \xi\rangle^s}\Big)
		=(\operatorname{sign}\xi \times \operatorname{sign}\eta)
		m_*(\xi,\eta),
	\end{equation*}
	where
	\begin{equation*}
		m_*(\xi,\eta)=
		\left(\frac{\langle \eta\rangle}{\langle \xi\rangle^s}\right)^{-(\alpha+\beta)}
		\chi\Big(\frac{\eta}{\langle \xi\rangle^s}\Big),
	\end{equation*}
	over the mixed Lebesgue space $L^p(\mathbb{R}_x;L^q(\mathbb{R}_v))$.
	
	We have already explained, in the proofs of Theorems \ref{result:1} and \ref{result:2}, how the Fourier multiplier $\operatorname{sign}\xi \times \operatorname{sign}\eta$ yields an operator that is bounded over any reflexive mixed Lebesgue space. This boundedness was justified by separating the action of the operator on each variable $x$ and $v$, and then interpreting mixed Lebesgue spaces as vector-valued Lebesgue spaces in order to apply boundedness theorems for vector-valued singular integral operators.
	
	Unfortunately, this natural strategy fails to establish the boundedness of $m_*(D_x,D_v)$, because the multiplier $m_*(\xi,\eta)$ mixes the variables $\xi$ and $\eta$. It seems that one may be able to overcome this hurdle by interpreting $m_*(D_x,D_v)$ as an operator-valued Fourier multiplier and then applying corresponding boundedness results (such as \cite[Theorem 5.5.10]{hvvl16} or \cite[Theorem 8.3.19]{hvvl17}). But this advanced technology requires some deep knowledge of the functional analysis of Banach spaces (such as the notion of $R$-boundedness of families of operators) and its use would be overkill in our context.
	
	Instead, here, we prefer to resort to \emph{ad hoc} theorems on the boundedness of Fourier multipliers. Specifically, we are now going to apply Corollary 1 from \cite{l70}, which generalizes the classical Marcinkiewicz multiplier theorem (see \cite[Corollary 6.2.5]{g14} for a recent presentation of this classical result) to mixed Lebesgue spaces. To this end, recall first that, for any given uniformly bounded function
	\begin{equation*}
		\psi(k):\mathbb{R}^N\to\mathbb{R} \quad \text{with }N\geq 1,
	\end{equation*}
	the Marcinkiewicz multiplier theorem states that the bound
	\begin{equation}\label{criterion:1}
		\sum_{\lambda\in\{0,1\}^N}\sup_{k\in\mathbb{R}^N}
		|k^\lambda\partial_k^\lambda \psi(k)|=
		\sum_{\lambda\in\{0,1\}^N}\sup_{k\in\mathbb{R}^N}
		|k_1^{\lambda_1}k_2^{\lambda_2}\cdots k_N^{\lambda_N} \partial_{k_1}^{\lambda_1}\partial_{k_2}^{\lambda_2}\cdots \partial_{k_N}^{\lambda_N}\psi(k)|<\infty
	\end{equation}
	is sufficient to guarantee the boundedness of the operator $\psi(D)$ over the Lebesgue spaces $L^p(\mathbb{R}^N)$, for any $1<p<\infty$. Corollary 1 in \cite{l70} generalizes this classical result and establishes that \eqref{criterion:1} also implies the boundedness of $\psi(D)$ over any mixed Lebesgue space $L^{p_1}_{k_1}L^{p_2}_{k_2}\cdots L^{p_N}_{k_N}$ on $\mathbb{R}^N$, with $1<p_1,p_2,\ldots,p_N<\infty$ (i.e., any reflexive mixed Lebesgue space).
	
	We should note that the classical H\"ormander--Mihlin criterion (see \cite[Theorem 5.2.7]{g14})
	\begin{equation}\label{criterion:2}
		\sum_{\substack{\lambda\in\mathbb{N}^N\\|\lambda|\leq [\frac N2]+1}}\sup_{k\in\mathbb{R}^N}
		|k|^{|\lambda|}|\partial_k^\lambda \psi(k)|<\infty,
	\end{equation}
	which also provides a sufficient condition for the boundedness of $\psi(D)$ over reflexive Lebesgue spaces, has been revisited more recently in \cite{ai16} where it is shown to entail the boundedness of $\psi(D)$ over reflexive mixed Lebesgue spaces, as well.
	
	Now, it is tedious but straighforward to verify that $m_*(\xi,\eta)$ satisfies \eqref{criterion:1} with $N=2$. This step requires that $\alpha+\beta\leq 0$. Perhaps the simplest strategy for this computation consists in noticing that one can write $m_*(\xi, \eta)=\psi(\xi,\langle\xi\rangle^{-s}\eta)$, where
	\begin{equation*}
		\psi(\xi,\eta)=\big(\langle\xi\rangle^{-2s}+\eta^2\big)^{-\frac{\alpha+\beta}2}\chi(\eta)
	\end{equation*}
	satisfies \eqref{criterion:1}, too. Observe, however, that $m_*(\xi,\eta)$ does not satisfy \eqref{criterion:2}. It therefore follows from a crucial application of Corollary 1 from \cite{l70} that $m_*(D_x,D_v)$ and $m_0(D_x,D_v)$ are bounded over $L^p(\mathbb{R}_x;L^q(\mathbb{R}_v))$.
	
	We may now deduce from \eqref{estimate:1} that
	\begin{equation}\label{estimate:2}
		\begin{aligned}
			\int_{\mathbb{R}}\frac{|\xi|}{\langle\xi\rangle}\langle\xi\rangle^{2\sigma}
			\big|\hspace{3pt}\widehat{\hspace{-3pt}\tilde f}(\xi)\big|^2 d\xi
			& \lesssim
			\left\|
				\langle D_x\rangle^{b}\langle D_v\rangle^{\beta}v\partial_xf
			\right\|_{L^{p'}_x L^{q'}_v}
			\left\|
				\langle D_x\rangle^{a}\langle D_v\rangle^{\alpha} f
			\right\|_{L^p_x L^q_v}
			\\
			&\quad +
			\left\|
				\langle D_x\rangle^{c}\langle D_v\rangle^{\gamma} f
			\right\|_{L^2_{x,v}}^2.
		\end{aligned}
	\end{equation}
	Finally, further observe that, for any $\varphi\in C^\infty_c(\mathbb{R})$ such that $\varphi(v)f(x,v)=f(x,v)$, one has
	\begin{equation*}
		|\hspace{3pt}\widehat{\hspace{-3pt}\tilde f}(\xi)|
		=\left|\int_{\mathbb{R}}\mathcal{F}_xf(\xi,v)\varphi(v) dv\right|
		\leq \big\|\langle D_v\rangle^{\gamma}\mathcal{F}_xf(\xi,v)\big\|_{L^2_v}
		\big\|\langle D_v\rangle^{-\gamma}\varphi\big\|_{L^2_v},
	\end{equation*}
	where $\mathcal{F}_x$ denotes the Fourier transform with respect to $x$ only, which leads to
	\begin{equation}\label{estimate:3}
		\int_{\mathbb{R}}\mathds{1}_{\{|\xi|\leq 1\}}
		\big|\hspace{3pt}\widehat{\hspace{-3pt}\tilde f}(\xi)\big|^2 d\xi
		\lesssim
		\int_{\mathbb{R}}\langle \xi\rangle^{2c}
		\big\|\langle D_v\rangle^{\gamma}\mathcal{F}_xf(\xi,v)\big\|_{L^2_v}^2 d\xi.
	\end{equation}
	All in all, combining \eqref{estimate:2} and \eqref{estimate:3}, we conclude that
	\begin{equation*}
		\begin{aligned}
			\int_{\mathbb{R}}\langle\xi\rangle^{2\sigma}
			\big|\hspace{3pt}\widehat{\hspace{-3pt}\tilde f}(\xi)\big|^2 d\xi
			& \lesssim
			\left\|
				\langle D_x\rangle^{b}\langle D_v\rangle^{\beta}v\partial_xf
			\right\|_{L^{p'}_x L^{q'}_v}
			\left\|
				\langle D_x\rangle^{a}\langle D_v\rangle^{\alpha} f
			\right\|_{L^p_x L^q_v}
			\\
			&\quad +
			\left\|
				\langle D_x\rangle^{c}\langle D_v\rangle^{\gamma} f
			\right\|_{L^2_{x,v}}^2,
		\end{aligned}
	\end{equation*}
	which completes the proof of the one-dimensional case.
	
	\emph{We focus now on higher dimensions $n\geq 2$.}
	To this end, we introduce a cutoff function with slightly different properties. Here, we are going to use some $\chi(r)\in\mathcal{S}(\mathbb{R}^n)$, with $\chi(0)=1$, such that its Fourier transform $\widehat \chi$ is compactly supported and 
	\begin{equation}\label{cutoff:1}
		|r|^{-2\gamma}\nabla\chi(r)\in L^\infty(\mathbb{R}^n).
	\end{equation}
	The latter property replaced the requirement that $\chi$ be constant near the origin in \eqref{cutoff:2}, which is now precluded by the compactness of the support of $\widehat \chi$ (because it implies the analyticity of $\chi$). Observe that such a cutoff function can be easily constructed by starting with some $\chi_0\in\mathcal{S}(\mathbb{R}^n)$ satisfying that $\chi_0(0)=1$ and $\widehat\chi_0$ is compactly supported, and then defining $\chi(r)=\chi_0(r)p(r)$ where $p(r)$ is a suitable polynomial with coefficients chosen so that $p(0)=1$ and $\partial^\lambda\chi(0)=0$, for all $\lambda\in\mathbb{N}^n$ with $1\leq |\lambda|\leq N$ and $N$ large enough.
	
	Thus, we generalize the one-dimensional strategy by considering now the multiplier
	\begin{equation*}
		m_1(\xi,\eta)=\frac{\operatorname{sign}\xi_1\times \operatorname{sign}\eta_1}{\langle \xi\rangle^{1-2\sigma}}
		\chi\Big(\frac{\eta}{\langle \xi\rangle^s}\Big)^2,
	\end{equation*}
	where $s\geq 0$ is given by \eqref{parameter:1}. In particular, we find that
	\begin{equation*}
		\xi\cdot\nabla_\eta m_1(\xi,\eta)=
		\frac{2|\xi_1|}{\langle\xi\rangle^{1-2\sigma}}\delta_0(\eta_1)\chi\Big(0,\frac{\eta'}{\langle \xi\rangle^s}\Big)^2
		+2\frac{\operatorname{sign}\xi_1\times \operatorname{sign}\eta_1}{\langle\xi\rangle^{1+s-2\sigma}}
		\xi\cdot\nabla\chi\Big(\frac{\eta}{\langle\xi\rangle^s}\Big)\chi\Big(\frac{\eta}{\langle \xi\rangle^s}\Big),
	\end{equation*}
	where $\eta'=(\eta_2,\ldots,\eta_n)\in\mathbb{R}^{n-1}$. Therefore, injecting $m_1(\eta,\xi)$ into the energy identity \eqref{energy:1} yields that
	\begin{equation*}
		\begin{aligned}
			\iint_{\mathbb{R}^n\times\mathbb{R}^{n-1}}\frac{|\xi_1|}{\langle\xi\rangle^{1-2\sigma}}
			\big|\chi\Big(0,\frac{\eta'}{\langle \xi\rangle^s}\Big)\widehat f(\xi,0,\eta')\big|^2 d\xi d\eta'
			\hspace{-40mm}
			\\
			& =(2\pi)^{2n}\operatorname{Re} \langle v\cdot\nabla_xf,m_1(D_x,D_v)f\rangle_{L^2_{x,v}}
			\\
			&\quad
			-\iint_{\mathbb{R}^n\times\mathbb{R}^n}
			\frac{\operatorname{sign}\xi_1\times \operatorname{sign}\eta_1}{\langle\xi\rangle^{1+s-2\sigma}}
			\xi\cdot\nabla\chi\Big(\frac{\eta}{\langle\xi\rangle^s}\Big)
			\chi\Big(\frac{\eta}{\langle\xi\rangle^s}\Big)
			\big|\widehat f(\xi,\eta)\big|^2 d\xi d\eta,
		\end{aligned}
	\end{equation*}
	whereby, further using \eqref{cutoff:1}, we deduce that
	\begin{equation}\label{estimate:4}
		\begin{aligned}
			\iint_{\mathbb{R}^n\times\mathbb{R}^{n-1}}\frac{|\xi_1|}{\langle\xi\rangle^{1-2\sigma}}
			\big|\chi\Big(0,\frac{\eta'}{\langle \xi\rangle^s}\Big)\widehat f(\xi,0,\eta')\big|^2 d\xi d\eta'
			\hspace{-40mm}
			\\
			&\lesssim
			\left\|
				\langle D_x\rangle^{b}\langle D_v\rangle^{\beta}v\cdot \nabla_xf
			\right\|_{L^{p'}_x L^{q'}_v}
			\left\|
				\frac{m_1(D_x,D_v)}{\langle D_x\rangle^{a+b}\langle D_v\rangle^{\alpha+\beta}}
				\langle D_x\rangle^{a}\langle D_v\rangle^{\alpha} f
			\right\|_{L^p_x L^q_v}
			\\
			&\quad
			+\iint_{\mathbb{R}^n\times\mathbb{R}^n}
			\frac{|\xi|\langle\xi\rangle^{2c}\langle\eta\rangle^{2\gamma}}{\langle\xi\rangle^{1+s(1+2\gamma)+2c-2\sigma}}
			\big|\widehat f(\xi,\eta)\big|^2 d\xi d\eta.
		\end{aligned}
	\end{equation}
	Note that $s(1+2\gamma)+2c-2\sigma=0$ in the last integrand above.
	
	Next, as in the one-dimensional case, we observe that the boundedness of the multiplier
	\begin{equation*}
		\begin{aligned}
			\frac{m_1(\xi,\eta)}{\langle \xi\rangle^{a+b}\langle \eta\rangle^{\alpha+\beta}}
			&=\frac{\operatorname{sign}\xi_1 \times \operatorname{sign}\eta_1}{\langle \xi\rangle^{1+a+b-2\sigma}\langle \eta\rangle^{\alpha+\beta}}
			\chi\Big(\frac{\eta}{\langle \xi\rangle^s}\Big)^2
			\\
			&=(\operatorname{sign}\xi_1 \times \operatorname{sign}\eta_1)
			\left(\frac{\langle \eta\rangle}{\langle \xi\rangle^s}\right)^{-(\alpha+\beta)}
			\chi\Big(\frac{\eta}{\langle \xi\rangle^s}\Big)^2,
		\end{aligned}
	\end{equation*}
	over the mixed Lebsegue space $L^p(\mathbb{R}_x;L^q(\mathbb{R}_v))$, follows from a direct application of Corollary~1 from \cite{l70}. Indeed, it is readily seen through a straightforward calculation that the above multiplier satisfies Marcinkiewicz's criterion \eqref{criterion:1}.
	
	Thus, in view of the boundedness of
	\begin{equation*}
		\frac{m_1(D_x,D_v)}{\langle D_x\rangle^{a+b}\langle D_v\rangle^{\alpha+\beta}},
	\end{equation*}
	we deduce from \eqref{estimate:4} that
	\begin{equation}\label{estimate:5}
		\begin{aligned}
			\int_{\mathbb{R}^n}\frac{|\xi_1|}{\langle\xi\rangle}\langle\xi\rangle^{2\sigma}
			\int_{\mathbb{R}^{n-1}}
			\big|\chi\Big(0,\frac{\eta'}{\langle \xi\rangle^s}\Big)\widehat f(\xi,0,\eta')\big|^2 d\eta' d\xi
			\hspace{-50mm}
			\\
			&\lesssim
			\left\|
				\langle D_x\rangle^{b}\langle D_v\rangle^{\beta}v\cdot \nabla_xf
			\right\|_{L^{p'}_x L^{q'}_v}
			\left\|
				\langle D_x\rangle^{a}\langle D_v\rangle^{\alpha} f
			\right\|_{L^p_x L^q_v}
			+\left\|
				\langle D_x\rangle^{c}\langle D_v\rangle^{\gamma} f
			\right\|_{L^2_{x,v}}^2.
		\end{aligned}
	\end{equation}
	There only remains to show now that the left-hand side above controls $\tilde f(x)$ in $H_x^\sigma(\mathbb{R}^n)$. To this end, observe that the Fourier transforms in $\eta'$ of both $\chi\big(0,{\eta'}/{\langle \xi\rangle^s}\big)$ and $\widehat f(\xi,0,\eta')$ are compactly supported. In particular, note that the support of the Fourier transform in $\eta'$ of $\chi\big(0,{\eta'}/{\langle \xi\rangle^s}\big)$ is uniformly bounded because $\langle \xi\rangle^s\geq 1$. Therefore, we conclude that the product $\chi\big(0,{\eta'}/{\langle \xi\rangle^s}\big)\widehat f(\xi,0,\eta')$ has a Fourier transform in $\eta'$ that is uniformly compactly supported, as well. Consequently, replicating estimate \eqref{support:1} based on Plancherel's theorem, we find that
	\begin{equation}\label{support:2}
		\begin{aligned}
			\big|\hspace{3pt}\widehat{\hspace{-3pt}\tilde f}(\xi)\big|^2
			&=
			\big|\chi\Big(0,\frac{0'}{\langle \xi\rangle^s}\Big)\widehat f(\xi,0,0')\big|^2
			=\frac 1{(2\pi)^{n-1}}\left|\int_{\mathbb{R}^{n-1}}
			\mathcal{F}_{\eta'}\Big[\chi\Big(0,\frac{\eta'}{\langle \xi\rangle^s}\Big)\widehat f(\xi,0,\eta')\Big] dv'\right|^2
			\\
			&\lesssim \frac 1{(2\pi)^{n-1}}
			\int_{\mathbb{R}^{n-1}}
			\left|\mathcal{F}_{\eta'}\Big[\chi\Big(0,\frac{\eta'}{\langle \xi\rangle^s}\Big)\widehat f(\xi,0,\eta')\Big]\right|^2 dv'
			=\int_{\mathbb{R}^{n-1}}
			\big|\chi\Big(0,\frac{\eta'}{\langle \xi\rangle^s}\Big)\widehat f(\xi,0,\eta')\big|^2 d\eta'.
		\end{aligned}
	\end{equation}
	All in all, combining \eqref{estimate:5} and \eqref{support:2}, and using that the kinetic transport equation is symmetric with respect to each coordinate of $\mathbb{R}^n$, we arrive at the estimate
	\begin{equation}\label{estimate:6}
		\begin{aligned}
			\int_{\mathbb{R}^n}\frac{|\xi|}{\langle\xi\rangle}\langle\xi\rangle^{2\sigma}
			\big|\hspace{3pt}\widehat{\hspace{-3pt}\tilde f}(\xi)\big|^2
			d\xi
			&\leq
			\sum_{j=1}^n
			\int_{\mathbb{R}^n}\frac{|\xi_j|}{\langle\xi\rangle}\langle\xi\rangle^{2\sigma}
			\big|\hspace{3pt}\widehat{\hspace{-3pt}\tilde f}(\xi)\big|^2
			d\xi
			\\
			&\lesssim
			\left\|
				\langle D_x\rangle^{b}\langle D_v\rangle^{\beta}v\cdot \nabla_xf
			\right\|_{L^{p'}_x L^{q'}_v}
			\left\|
				\langle D_x\rangle^{a}\langle D_v\rangle^{\alpha} f
			\right\|_{L^p_x L^q_v}
			\\
			&\quad +\left\|
				\langle D_x\rangle^{c}\langle D_v\rangle^{\gamma} f
			\right\|_{L^2_{x,v}}^2.
		\end{aligned}
	\end{equation}
	
	Finally, it is readily seen that estimate \eqref{estimate:3} can be reproduced in higher dimensions, thereby showing that
	\begin{equation}\label{estimate:7}
		\int_{\mathbb{R}^n}\mathds{1}_{\{|\xi|\leq 1\}}
		\big|\hspace{3pt}\widehat{\hspace{-3pt}\tilde f}(\xi)\big|^2 d\xi
		\lesssim
		\left\|
			\langle D_x\rangle^{c}\langle D_v\rangle^{\gamma} f
		\right\|_{L^2_{x,v}}^2.
	\end{equation}
	The proof is then easily concluded by combining \eqref{estimate:6} and \eqref{estimate:7}.
\end{proof}

\begin{rema}
	Theorem \ref{result:3} and its ensuing corollaries below somewhat differ from the contributions from \cite{jlt20}. More precisely, while Theorem 2 in \cite{jlt20} and Theorem \ref{result:3} yield a similar gain of regularity on the velocity averages in the range of parameters where they intersect, they also apply to distinct settings which are not simultaneously covered by both results. Indeed, in Theorem \ref{result:3}, we consider a broader range of possible derivatives in location $x$ and velocity $v$, both in the solution and the source term, whereas Theorem 2 in \cite{jlt20} remains more focused on reaching a wider range of integrability conditions which do not necessarily rely on duality between the left and right-hand sides of the kinetic transport equation.
\end{rema}

\begin{rema}
	Theorem \ref{result:3} only considers the case of mixed Lebesgue spaces $L^p(\mathbb{R}^n_x;L^q(\mathbb{R}^n_v))$. However, observe that the above proof can be adapted to spaces of the type $L^p(\mathbb{R}^n_v;L^q(\mathbb{R}^n_x))$ and follows through \emph{mutatis mutandis.}
	
	The preceding proof can also be easily adapted to handle homogeneous Sobolev spaces by employing Riesz potentials $|D_x|$ and $|D_v|$ instead of Bessel potentials $\langle D_x\rangle$ and $\langle D_v\rangle$, with appropriate restrictions on the range of regularity parameters.
\end{rema}

Observe that Theorem \ref{result:3} recovers the exact same gain of regularity as the Hilbertian setting provided we take $c=a$ and $\gamma=\alpha$ therein. A precise statement is given in the next corollary. Again, this new version of velocity averaging generalizes classical results with a duality principle enabled by the energy method. More precisely, it establishes that the gain of regularity on velocity averages is preserved when one loses some local integrability on the source term $v\cdot\nabla_x f$, provided a corresponding integrability is gained on the density $f$. However, it should be noted that it fails to recover the full range of regularity parameters made available in \cite[Theorem 4.5]{am14}. Thus, it remains unclear whether the same duality principle remains valid if $-\alpha<\beta<\frac 12$.

\begin{cor}\label{result:5}
	Let the family $\{f_\lambda(x,v)\}_{\lambda\in\Lambda}\subset\mathcal{S}(\mathbb{R}^n\times\mathbb{R}^n)$ be uniformly compactly supported in $v$ and such that the following subsets are bounded
	\begin{equation*}
		\begin{aligned}
			\big\{(1-\Delta_x)^\frac a 2(1-\Delta_v)^\frac \alpha 2f_\lambda\big\}_{\lambda\in\Lambda} & \subset L^p(\mathbb{R}^n_x;L^q(\mathbb{R}^n_v))\cap L^2(\mathbb{R}^n_x\times\mathbb{R}^n_v),
			\\
			\big\{(1-\Delta_x)^\frac b 2(1-\Delta_v)^\frac \beta 2v\cdot\nabla_x f_\lambda\big\}_{\lambda\in\Lambda} & \subset L^{p'}(\mathbb{R}^n_x;L^{q'}(\mathbb{R}^n_v)),
		\end{aligned}
	\end{equation*}
	for some given $1<p,q<\infty$ and any regularity parameters $a,b,\alpha,\beta\in\mathbb{R}$ satisfying the constraints
	\begin{equation*}
		1+b-a\geq 0,
		\quad
		\alpha+\beta\leq 0,
		\quad
		\alpha>-\frac 12.
	\end{equation*}
	Then $\{\tilde f_\lambda\}_{\lambda\in\Lambda}$ is a bounded family in $H^\sigma(\mathbb{R}^n_x)$, where
	\begin{equation*}
		\sigma=
		\frac{1+b-a}{1+\alpha-\beta}\Big(\frac 12+\alpha\Big)+a.
	\end{equation*}
	
	In particular, if $2\leq p,q<\infty$ and $\{f_\lambda(x,v)\}_{\lambda\in\Lambda}$ is uniformly compactly supported in all variables $(x,v)$, then the uniform bounds
	\begin{equation*}
		\begin{aligned}
			\big\{(1-\Delta_x)^\frac a 2(1-\Delta_v)^\frac \alpha 2f_\lambda\big\}_{\lambda\in\Lambda} & \subset L^p(\mathbb{R}^n_x;L^q(\mathbb{R}^n_v)),
			\\
			\big\{(1-\Delta_x)^\frac b 2(1-\Delta_v)^\frac \beta 2v\cdot\nabla_x f_\lambda\big\}_{\lambda\in\Lambda} & \subset L^{p'}(\mathbb{R}^n_x;L^{q'}(\mathbb{R}^n_v)),
		\end{aligned}
	\end{equation*}
	are sufficient to imply that $\{\tilde f_\lambda\}_{\lambda\in\Lambda}$ is a bounded family in $H^\sigma(\mathbb{R}^n_x)$.
\end{cor}

\begin{proof}
	The first part of the corollary is a mere restatement of Theorem \ref{result:3} where we set $c=a$ and $\gamma=\alpha$. Therefore, we only need to focus on the localization principle provided by the second part of the corollary, which follows directly from the bound
	\begin{equation}\label{localization:1}
		\|\langle D\rangle^r h\|_{L^{p_0}(\mathbb{R}^n)}\leq C\|\langle D\rangle^r h\|_{L^{p_1}(\mathbb{R}^n)},
	\end{equation}
	for any $h\in C_c^\infty(\mathbb{R}^n)$, $r\in\mathbb{R}$ and $1\leq p_0 \leq p_1\leq\infty$,
	where the constant $C>0$ depends on the size of the support of $h$ and other fixed parameters. Note that this bound holds for scalar-valued and vector-valued Lebesgue spaces (the proof is the same).
	
	For the sake of completeness, we give now a short justification of \eqref{localization:1}. To this end, notice first that \eqref{localization:1} is obvious when $r$ is a positive even integer, for $\langle D\rangle^r=(1-\Delta)^\frac r2$ is a local differential operator in this case. For general values $r\in\mathbb{R}$, one can always decompose
	\begin{equation*}
		\langle D\rangle^r h=\langle D\rangle^{r-2N} (1-\Delta)^N h,
	\end{equation*}
	for some large positive integer $N$. Therefore, we only need to establish \eqref{localization:1} for negative values $r<0$.
	
	There are several ways to proceed. Perhaps the simplest option consists in considering the convolution form of Bessel potentials $\langle D\rangle^{-s}$, with $s>0$. Indeed, it is relatively easy to establish that
	\begin{equation}\label{kernel:1}
		\langle D\rangle^{-s}h(x)=\int_{\mathbb{R}^n} G_s(x-y)h(y)dy,
	\end{equation}
	where
	\begin{equation*}
		G_s(x)=\frac{1}{(4\pi)^\frac n2\Gamma(\frac s2)}\int_0^\infty e^{-t}e^{-\frac{|x|^2}{4t}}t^\frac{s-n}2 \frac {dt}t
	\end{equation*}
	and $\Gamma(\frac s2)=\int_0^\infty e^{-t}t^\frac s2\frac{dt}t$ denotes the usual Gamma function.
	In particular, observe that $G_s(x)$ is integrable and smooth away from the origin.
	Moreover, one can show, for some constant $C_{s,n}>0$, that
	\begin{equation*}
		G_s(x)\leq C_{s,n}e^{-\frac{|x|}2},
	\end{equation*}
	whenever $|x|\geq 2$. We refer to \cite[Section 6.1.2]{g14:2} for more details on Bessel potentials.
	
	Thus, we may now decompose the kernel into $G_s=G_s^0+G_s^1$, where $G_s^0\in L^1(\mathbb{R}^n)$ is compactly supported and $G_s^1\in\mathcal{S}(\mathbb{R}^n)$ is arbitrarily small in the sense that
	\begin{equation*}
		\|\langle D\rangle^{s}G_s^1\|_{L^1(\mathbb{R}^n)}\leq \delta,
	\end{equation*}
	for some $\delta>0$ to be determined later on. Then, since $G_s^0* h$ is also compactly supported, we deduce that
	\begin{equation*}
		\begin{aligned}
			\|\langle D\rangle^{-s} h\|_{L^{p_0}(\mathbb{R}^n)} & \leq \|G_s^0* h\|_{L^{p_0}(\mathbb{R}^n)}+\|G_s^1* h\|_{L^{p_0}(\mathbb{R}^n)}
			\\
			& \lesssim \|G_s^0* h\|_{L^{p_1}(\mathbb{R}^n)}+\|G_s^1* h\|_{L^{p_0}(\mathbb{R}^n)}
			\\
			& \leq \|\langle D\rangle^{-s} h\|_{L^{p_1}(\mathbb{R}^n)}+\|G_s^1* h\|_{L^{p_1}(\mathbb{R}^n)}+\|G_s^1* h\|_{L^{p_0}(\mathbb{R}^n)}
			\\
			& = \|\langle D\rangle^{-s} h\|_{L^{p_1}(\mathbb{R}^n)}+\|\langle D\rangle^{s}G_s^1* \langle D\rangle^{-s}h\|_{L^{p_1}(\mathbb{R}^n)}+\|\langle D\rangle^{s}G_s^1* \langle D\rangle^{-s}h\|_{L^{p_0}(\mathbb{R}^n)}
			\\
			& \leq \|\langle D\rangle^{-s} h\|_{L^{p_1}(\mathbb{R}^n)}
			+\delta
			\left(
			\|\langle D\rangle^{-s}h\|_{L^{p_1}(\mathbb{R}^n)}
			+\|\langle D\rangle^{-s}h\|_{L^{p_0}(\mathbb{R}^n)}
			\right).
		\end{aligned}
	\end{equation*}
	Finally, choosing $\delta>0$ sufficently small, we deduce that the local embedding \eqref{localization:1} holds for negative regularity indices, and therefore for all $r\in\mathbb{R}$. This concludes the proof of the corollary.
\end{proof}

\begin{rema}
	Observe that Corollary \ref{result:5} holds in the range $1< p,q \leq 2$, as well. In this case, further assuming that $f_\lambda(x,v)$ is compactly supported in all variables $(x,v)$, one can deduce that the uniform bounds
	\begin{equation*}
		\begin{aligned}
			(1-\Delta_x)^\frac a 2(1-\Delta_v)^\frac \alpha 2f_\lambda & \in L^2(\mathbb{R}^n_x\times\mathbb{R}^n_v),
			\\
			(1-\Delta_x)^\frac b 2(1-\Delta_v)^\frac \beta 2v\cdot\nabla_x f_\lambda & \in L^{p'}(\mathbb{R}^n_x;L^{q'}(\mathbb{R}^n_v)),
		\end{aligned}
	\end{equation*}
	are enough to imply that $\tilde f_\lambda$ belongs to $H^\sigma(\mathbb{R}^n_x)$ uniformly. But $L^{p'}_xL^{q'}_v\subset L^2_{x,v}$ over compact domains and, therefore, this local result is strictly weaker than the corresponding classical Hilbertian case $p=q=2$. We conclude that, in the range $1<p,q < 2$, Corollary \ref{result:3} is a global result and only provides new information on $\tilde f_\lambda$ if its support is unbounded.
\end{rema}

The following result is another direct consequence of Theorem \ref{result:3}. Specifically, we use now Sobolev embedding theorems to compare and relate the bounds
\begin{equation*}
	(1-\Delta_x)^\frac a 2(1-\Delta_v)^\frac \alpha 2f \in L^p(\mathbb{R}^n_x;L^q(\mathbb{R}^n_v))
\end{equation*}
and
\begin{equation*}
	(1-\Delta_x)^\frac c 2(1-\Delta_v)^\frac \gamma 2f \in L^2(\mathbb{R}^n_x\times\mathbb{R}^n_v),
\end{equation*}
thereby simplifying the statement of Theorem \ref{result:3} and establish more practical versions of velocity averaging results. In this case, a new duality principle shows how to trade integrability and regularity in a way that preserves the regularity of velocity averages obtained from the Hilbertian setting $p=q=2$.

\begin{cor}\label{result:4}
	Let the family $\{f_\lambda(x,v)\}_{\lambda\in\Lambda}\subset\mathcal{S}(\mathbb{R}^n\times\mathbb{R}^n)$ be uniformly compactly supported in $v$ and such that the following subsets are bounded
	\begin{equation*}
		\begin{aligned}
			\big\{(1-\Delta_x)^{\frac 1 2\left(a+n\left(\frac 1p-\frac 12\right)\right)}(1-\Delta_v)^{\frac 1 2\left(\alpha+n\left(\frac 1q-\frac 12\right)\right)}f_\lambda\big\}_{\lambda\in\Lambda} & \subset L^p(\mathbb{R}^n_x;L^q(\mathbb{R}^n_v)),
			\\
			\big\{(1-\Delta_x)^{\frac 1 2\left(b-n\left(\frac 1p-\frac 12\right)\right)}(1-\Delta_v)^{\frac 1 2\left(\beta-n\left(\frac 1q-\frac 12\right)\right)}v\cdot\nabla_x f_\lambda\big\}_{\lambda\in\Lambda} & \subset L^{p'}(\mathbb{R}^n_x;L^{q'}(\mathbb{R}^n_v)),
		\end{aligned}
	\end{equation*}
	for some given $1<p,q\leq 2$ and any regularity parameters $a,b,\alpha,\beta\in\mathbb{R}$ satisfying the constraints
	\begin{equation*}
		1+b-a\geq 0,
		\quad
		\alpha+\beta\leq 0,
		\quad
		\alpha>-\frac 12.
	\end{equation*}
	Then $\{\tilde f_\lambda\}_{\lambda\in\Lambda}$ is a bounded family in $H^\sigma(\mathbb{R}^n_x)$, where
	\begin{equation*}
		\sigma=
		\frac{1+b-a}{1+\alpha-\beta}\Big(\frac 12+\alpha\Big)+a.
	\end{equation*}
\end{cor}

\begin{proof}
	This result follows from a direct combination of Theorem \ref{result:3} with the Sobolev embedding
	\begin{equation*}
		W^{a+n\left(\frac 1p-\frac 12\right),p}(\mathbb{R}^n_x;W^{\alpha+n\left(\frac 1q-\frac 12\right),q}(\mathbb{R}^n_v))
		\subset
		H^{a}(\mathbb{R}^n_x;H^{\alpha}(\mathbb{R}^n_v)).
	\end{equation*}
	A few words concerning the validity of the above continuous embedding are in order. Even though the injection
	\begin{equation*}
		W^{\alpha+n\left(\frac 1q-\frac 12\right),q}(\mathbb{R}^n)
		\subset
		H^{\alpha}(\mathbb{R}^n),
		\quad\text{with }1<q<2\text{ and }\alpha\in\mathbb{R},
	\end{equation*}
	is absolutely classical and does not require any further explanation, we feel that, for the sake of completeness, a short justification of the vector-valued Sobolev embedding
	\begin{equation}\label{embedding:1}
		W^{a+n\left(\frac 1p-\frac 12\right),p}(\mathbb{R}^n;X)
		\subset
		H^{a}(\mathbb{R}^n;X),
		\quad\text{with }1<p<2\text{ and }a\in\mathbb{R},
	\end{equation}
	for any given Banach space, is required here.
	
	We believe that the simplest way to extend standard scalar-valued Sobolev embeddings to their vector-valued analogue consists in considering the convolution form \eqref{kernel:1} of Bessel potentials $\langle D\rangle^{-s}$, with $s>0$, which remains valid for any $h\in\mathcal{S}(\mathbb{R}^n;X)$.
	
	Indeed, notice first that scalar-valued Sobolev embeddings are obtained as a direct consequence of the Hardy--Littlewood--Sobolev inequality upon noticing that $G_s(x)$ is smooth away from the origin, decays exponentially and, further assuming $0<s<n$, that
	\begin{equation*}
		\lim_{x\to 0}|x|^{n-s}G_s(x)=
		\frac{1}{(4\pi)^\frac n2\Gamma(\frac s2)}\int_0^\infty e^{-\frac{1}{4t}}t^\frac{s-n}2 \frac {dt}t.
	\end{equation*}
	In the vector-valued setting, one may also deduce, from the non-negativeness of $G_s$, for all $s>0$, that
	\begin{equation*}
		\begin{aligned}
			\langle D\rangle^{s}\|h\|_X
			=
			\langle D\rangle^{s}\|\langle D\rangle^{-s}\langle D\rangle^{s}h\|_X
			\leq
			\langle D\rangle^{s}\langle D\rangle^{-s}\|\langle D\rangle^{s}h\|_X
			=
			\|\langle D\rangle^{s}h\|_X,
		\end{aligned}
	\end{equation*}
	thereby showing that $\langle D\rangle^{s}\|h\|_X\in L^p(\mathbb{R}^n)$, whenever $\langle D\rangle^{s}h\in L^p(\mathbb{R}^n;X)$. Thus, we conclude that the vector-valued Sobolev spaces inherit the embeddings from the corresponding scalar-valued case, which implies \eqref{embedding:1}.
	The proof of the corollary is now complete.
\end{proof}

\begin{rema}
	We notice that, by taking $b=-a$ and $\beta=-\alpha$ in Corollaries \ref{result:5} and \ref{result:4}, one recovers a regularity index $\sigma=\frac 12$. This is on par with the result \eqref{duality:1} deduced at the end of Section \ref{section:2} from the duality principle in the energy method.
\end{rema}

\begin{rema}
	One may also be interested in the implication of Sobolev embeddings on Theorem \ref{result:3} when $2\leq p,q<\infty$. In this case, we deduce by combining Sobolev embeddings with Theorem~\ref{result:3} that, if $f_\lambda(x,v)$ is compactly supported in $v$ and satisfies the uniform bounds
	\begin{equation*}
		\begin{aligned}
			f_\lambda & \in
			H^{a}(\mathbb{R}^n_x;H^{\alpha}(\mathbb{R}^n_v)),
			\\
			v\cdot\nabla_x f_\lambda & \in
			W^{b+n(\frac 12-\frac 1p),p'}(\mathbb{R}^n_x;W^{\beta+n(\frac 12-\frac 1q),q'}(\mathbb{R}^n_v))
			\subset
			H^{b}(\mathbb{R}^n_x;H^{\beta}(\mathbb{R}^n_v)),
		\end{aligned}
	\end{equation*}
	where
	\begin{equation*}
		1+b-a\geq 0,
		\quad
		\alpha+\beta\leq 0,
		\quad
		\alpha>-\frac 12,
	\end{equation*}
	then $\tilde f_\lambda$ belongs uniformly to $H^\sigma(\mathbb{R}^n_x)$, with
	\begin{equation*}
		\sigma=
		\frac{1+b-a}{1+\alpha-\beta}\Big(\frac 12+\alpha\Big)+a.
	\end{equation*}
	However, this result is not new. Indeed, it is obviously the strongest when $p=q=2$ and, therefore, it is already contained in the classical Hilbertian case $p=q=2$.
\end{rema}

For the sake of completeness, we record in the following two corollaries what happens to Theorem \ref{result:3} in the cases $p\leq 2\leq q$ and $q\leq 2\leq p$, respectively, by combining the methods of Corollaries \ref{result:5} and \ref{result:4}.

\begin{cor}\label{result:6}
	Let the family $\{f_\lambda(x,v)\}_{\lambda\in\Lambda}\subset\mathcal{S}(\mathbb{R}^n\times\mathbb{R}^n)$ be uniformly compactly supported in $v$ and such that the following subsets are bounded
	\begin{equation*}
		\begin{aligned}
			\big\{(1-\Delta_x)^{\frac 1 2\left(a+n\left(\frac 1p-\frac 12\right)\right)}(1-\Delta_v)^{\frac \alpha 2}f_\lambda\big\}_{\lambda\in\Lambda} & \subset L^p(\mathbb{R}^n_x;L^q(\mathbb{R}^n_v)),
			\\
			\big\{(1-\Delta_x)^{\frac 1 2\left(b-n\left(\frac 1p-\frac 12\right)\right)}(1-\Delta_v)^{\frac \beta 2}v\cdot\nabla_x f\big\}_{\lambda\in\Lambda} & \subset L^{p'}(\mathbb{R}^n_x;L^{q'}(\mathbb{R}^n_v)),
		\end{aligned}
	\end{equation*}
	for some given $1<p\leq 2\leq q<\infty$ and any regularity parameters $a,b,\alpha,\beta\in\mathbb{R}$ satisfying the constraints
	\begin{equation*}
		1+b-a\geq 0,
		\quad
		\alpha+\beta\leq 0,
		\quad
		\alpha>-\frac 12.
	\end{equation*}
	Then $\{\tilde f_\lambda\}_{\lambda\in\Lambda}$ is a bounded family in $H^\sigma(\mathbb{R}^n_x)$, where
	\begin{equation*}
		\sigma=
		\frac{1+b-a}{1+\alpha-\beta}\Big(\frac 12+\alpha\Big)+a.
	\end{equation*}
\end{cor}

\begin{proof}
	This result follows from a direct combination of Theorem \ref{result:3} with the embeddings
	\begin{equation*}
		W^{a+n\left(\frac 1p-\frac 12\right),p}(\mathbb{R}^n_x;W^{\alpha,q}(K_v))
		\subset
		W^{a+n\left(\frac 1p-\frac 12\right),p}(\mathbb{R}^n_x;H^{\alpha}(K_v))
		\subset
		H^{a}(\mathbb{R}^n_x;H^{\alpha}(K_v)),
	\end{equation*}
	for any compact subset $K_v\subset\mathbb{R}^n_v$, which is a Sobolev embedding in $x$ coupled with the local embedding \eqref{localization:1} of Lebesgue spaces in $v$.
\end{proof}

\begin{cor}\label{result:7}
	Let the family $\{f_\lambda(x,v)\}_{\lambda\in\Lambda}\subset\mathcal{S}(\mathbb{R}^n\times\mathbb{R}^n)$ be uniformly compactly supported in all variables $(x,v)$ and such that the following subsets are bounded
	\begin{equation*}
		\begin{aligned}
			\big\{(1-\Delta_x)^{\frac a 2}(1-\Delta_v)^{\frac 1 2\left(\alpha+n\left(\frac 1q-\frac 12\right)\right)}f_\lambda\big\}_{\lambda\in\Lambda} & \subset L^p(\mathbb{R}^n_x;L^q(\mathbb{R}^n_v)),
			\\
			\big\{(1-\Delta_x)^{\frac b 2}(1-\Delta_v)^{\frac 1 2\left(\beta-n\left(\frac 1q-\frac 12\right)\right)}v\cdot\nabla_x f_\lambda\big\}_{\lambda\in\Lambda} & \subset L^{p'}(\mathbb{R}^n_x;L^{q'}(\mathbb{R}^n_v)),
		\end{aligned}
	\end{equation*}
	for some given $1<q\leq 2\leq p<\infty$ and any regularity parameters $a,b,\alpha,\beta\in\mathbb{R}$ satisfying the constraints
	\begin{equation*}
		1+b-a\geq 0,
		\quad
		\alpha+\beta\leq 0,
		\quad
		\alpha>-\frac 12.
	\end{equation*}
	Then $\{\tilde f_\lambda\}_{\lambda\in\Lambda}$ is a bounded family in $H^\sigma(\mathbb{R}^n_x)$, where
	\begin{equation*}
		\sigma=
		\frac{1+b-a}{1+\alpha-\beta}\Big(\frac 12+\alpha\Big)+a.
	\end{equation*}
\end{cor}

\begin{proof}
	This result follows from a direct combination of Theorem \ref{result:3} with the embeddings
	\begin{equation*}
		W^{a,p}(K_x;W^{\alpha+n\left(\frac 1q-\frac 12\right),q}(\mathbb{R}^n_v))
		\subset
		H^{a}(K_x;W^{\alpha+n\left(\frac 1q-\frac 12\right),q}(\mathbb{R}^n_v))
		\subset
		H^{a}(K_x;H^{\alpha}(\mathbb{R}^n_v)),
	\end{equation*}
	for any compact subset $K_x\subset\mathbb{R}^n_x$, which is a Sobolev embedding in $v$ coupled with the local embedding \eqref{localization:1} of Lebesgue spaces in $x$.
\end{proof}

% ==============
% = Dispersion =
% ==============
\section{Dispersion and the energy method}
\label{section:dispersion}

In this last section, in order to further expand the range of applicability of the energy method, we explore the effects of kinetic dispersion on velocity averaging lemmas.

The dispersive effects in kinetic transport equations provide a vast array of estimates on the transport flow in mixed Lebesgue spaces. The basic kinetic dispersive estimate established in \cite{cp96} states that
\begin{equation}\label{dispersion:1}
	\|f(x-tv,v)\|_{L^p(\mathbb{R}^n_x;L^r(\mathbb{R}^n_v))}
	\leq \frac 1{|t|^{n(\frac 1r-\frac 1p)}}
	\|f(x,v)\|_{L^r(\mathbb{R}^n_x;L^p(\mathbb{R}^n_v))},
\end{equation}
for all $0< r\leq p\leq \infty$ and $t\neq 0$. It shows that the kinetic transport flow can be controlled in some mixed Lebesgue spaces for positive times, even though the data might not belong to these spaces initially.

As was also established in \cite{cp96} and further extended in \cite{kt98}, combining the preceding dispersive estimate with a $TT^*$-argument leads to the Strichartz estimate for the transport flow
\begin{equation}\label{dispersion:2}
	\|f(x-tv,v)\|_{L^q(\mathbb{R}_t;L^p(\mathbb{R}^n_x;L^r(\mathbb{R}^n_v)))}
	\leq C
	\|f(x,v)\|_{L^a(\mathbb{R}^n_x\times \mathbb{R}^n_v)},
\end{equation}
where $C>0$ is a fixed constant and the parameters $0< a,p,q,r\leq\infty$ satisfy either $a=p=q=r=\infty$ or
\begin{equation*}
	\frac 2q=n\left(\frac1r-\frac 1p\right),\quad \frac 2a=\frac 1p+\frac 1r
	\quad\text{and}\quad a<q
	\text{ (that is $\textstyle p<\frac{n+1}{n-1}r$).}
\end{equation*}
Note that the endpoint case $a=q$ is false in every dimension $n\geq 1$, unless $a=q=\infty$, as shown in \cite{bbgl14}.

Naturally, by interpolating \eqref{dispersion:1} and \eqref{dispersion:2}, it is then possible to obtain a wider range of dispersive estimates. This was performed in \cite[Corollary 8.2]{o11} where a careful interpolation procedure shows that
\begin{equation}\label{dispersion:3}
	\|f(x-tv,v)\|_{L^{q,c}(\mathbb{R}_t;L^p(\mathbb{R}^n_x;L^r(\mathbb{R}^n_v)))}
	\leq C
	\|f(x,v)\|_{L^b(\mathbb{R}^n_x;L^c(\mathbb{R}^n_v))},
\end{equation}
where $L^{q,c}$ denotes the usual Lorentz (quasi-Banach) space, $C>0$ is a fixed constant and the parameters $0< b,c,p,q,r\leq\infty$ satisfy
\begin{equation*}
	\frac 1q=n\left(\frac1b-\frac 1p\right),\quad \frac 1p+\frac 1r=\frac1b+\frac 1c,
	\quad 0<r<b<c<p<\infty,
	\quad\text{and}\quad p<\frac n{n-1}c.
\end{equation*}
Recall that $L^{q,c}\subset L^q$ whenever $c\leq q$, which is equivalent to $c\leq\frac{n+1}nr$ in the present setting. We refer the reader to \cite[Section 1.4]{g14} for more details on Lorentz spaces.

\medskip
The aforementioned kinetic dispersive phenomena have important implications on the analysis of the evolution of kinetic equations on the whole Euclidean space. For instance, as shown in \cite{a11}, under suitable assumptions, it is possible to establish the existence of global solutions to the Boltzmann equation \eqref{boltzmann} by exploiting the dispersion of the kinetic flow with some fine regularizing properties of the Boltzmann collision operator.

Kinetic dispersion has also already been previously employed in the context of velocity averaging lemmas in \cite{am14} and \cite{am19}. The methods therein are based on the use of a suitable parametrix representation formula. We are now going to adopt a similar approach and, therefore, recall the basic principles leading to such parametrices and establish their dispersive properties.

To this end, let us consider a cutoff function $\chi_0(t)\in C^\infty_c(\mathbb{R})$ such that $\int_\mathbb{R}\chi_0(t)dt=1$ and define
\begin{equation*}
	\chi_1(t)=\mathds{1}_{\{t\geq 0\}}-\int_{-\infty}^t\chi_0(s)ds.
\end{equation*}
The cutoff $\chi_1(t)$ is obviously not smooth. However, note that it is bounded, compactly supported and $\int_{\mathbb{R}}\chi_1(t)dt=\int_{\mathbb{R}}t\chi_0(t)dt$. Now, for any $f(x,v)\in\mathcal{S}(\mathbb{R}^n_x\times\mathbb{R}^n_v)$, a straightfoward integration by parts (if $v\neq 0$; the identity is straightforward when $v=0$) yields that
\begin{equation*}
	\begin{aligned}
		\int_\mathbb{R}v\cdot\nabla_xf(x-tv,v)\chi_1(t)dt
		&=
		-\int_\mathbb{R}\frac d{dt}f(x-tv,v)\chi_1(t)dt
		\\
		&=f(x,v)-
		\int_\mathbb{R}f(x-tv,v)\chi_0(t)dt,
	\end{aligned}
\end{equation*}
whereby we deduce the parametrix representation formula
\begin{equation}\label{parametrix:1}
	f(x,v)=\int_\mathbb{R}f(x-tv,v)\chi_0(t)dt+\int_\mathbb{R}v\cdot\nabla_xf(x-tv,v)\chi_1(t)dt.
\end{equation}

The next two lemmas summarize the precise dispersive properties of \eqref{parametrix:1} that will be used in combination with the energy method later on.

\begin{lem}\label{dispersive:1}
	Let $1\leq r_0\leq p_0\leq\infty$ and $1\leq r_1\leq p_1\leq\infty$ be such that
	\begin{equation*}
		\frac1{p_0}+\frac1{r_0}=\frac1{p_1}+\frac1{r_1}
		\qquad
		\text{and}
		\qquad
		\frac n{r_1}<1+\frac n{p_0}.
	\end{equation*}
	Further suppose that
	\begin{equation}\label{assumption:1}
		\frac{n-1}{p_1}<\frac n{p_0}\leq \frac n{p_1},
	\end{equation}
	or
	\begin{equation}\label{assumption:2}
		\frac{n-1}{r_0'}<\frac n{r_1'}\leq \frac n{r_0'}.
	\end{equation}
	Observe that \eqref{assumption:1} and \eqref{assumption:2} cannot hold simultaneously unless $p_0=p_1$ and $r_0=r_1$.
	
	Then, for any compactly supported $\chi(t)\in L^\infty(\mathbb{R})$, one has the dispersive estimate
	\begin{equation*}
		\left\|\int_\mathbb{R}f(x-tv,v)\chi(t)dt\right\|_{L^{p_0}_xL^{r_0}_v}
		\leq C\left\|f(x,v)\right\|_{L^{r_1}_xL^{p_1}_v},
	\end{equation*}
	for some constant $C>0$ that only depends on $\chi$ and fixed parameters.
\end{lem}

\begin{proof}
	We distinguish four cases
	\begin{enumerate}
		\item $p_0=p_1$ and $r_0=r_1$,
		\item $1\leq r_0<r_1= p_1<p_0\leq\infty$,
		\item $1\leq r_0<r_1< p_1<p_0\leq\infty$,
		\item $1\leq r_1<r_0\leq p_0<p_1\leq\infty$,
	\end{enumerate}
	which we treat separately.
	
	\emph{We begin with the case $p_0=p_1$ and $r_0=r_1$.} A direct application of the dispersion estimate \eqref{dispersion:1} yields that
	\begin{equation}\label{estimate:8}
		\left\|\int_\mathbb{R}f(x-tv,v)\chi(t)dt\right\|_{L^{p_0}_xL^{r_0}_v}
		\leq \bigg\|\frac{\chi(t)}{|t|^{n\left(\frac1{r_0}-\frac1{p_0}\right)}}\bigg\|_{L^1}
		\left\|f(x,v)\right\|_{L^{r_0}_xL^{p_0}_v},
	\end{equation}
	which concludes the justification of this case upon noticing that $n\left(\frac1{r_0}-\frac1{p_0}\right)<1$.
	
	Observe that the range of applicability of the parameters $r_0$ and $p_0$ could be greatly expanded in the preceding estimate by supposing that the support of $\chi$ does not contain the origin. This observation is exploited in Lemma \ref{dispersive:2} below, where a wider array of dispersive estimates is provided under some restrictions on the cutoff $\chi(t)$.
	
	\emph{Next, we focus on the setting $1\leq r_0<r_1= p_1<p_0\leq\infty$.} This case necessarily falls in the range of parameters satisfying \eqref{assumption:1}. Denoting here
	\begin{equation*}
		\frac 2q=n\left(\frac1{r_0}-\frac 1{p_0}\right)
		\qquad\text{and}\qquad
		\frac 2a=\frac 1{p_0}+\frac 1{r_0}=\frac 2{p_1}=\frac 2{r_1},
	\end{equation*}
	we then observe that $a<q$ is equivalent to $\frac{n-1}{p_1}<\frac n{p_0}$ (note also that $\frac n{r_1}<1+\frac n{p_0}$ is now equivalent to $q>1$, which is a weaker constraint because $a>1$).
	Therefore, by virtue of the Strichartz estimate \eqref{dispersion:2}, we deduce that
	\begin{equation}\label{estimate:9}
		\left\|\int_\mathbb{R}f(x-tv,v)\chi(t)dt\right\|_{L^{p_0}_xL^{r_0}_v}
		\leq \|\chi\|_{L^{q'}}
		\left\|f(x-tv,v)\right\|_{L^q_tL^{p_0}_xL^{r_0}_v}
		\lesssim
		\left\|f(x,v)\right\|_{L^{a}_{x,v}},
	\end{equation}
	which completes the proof for this case.
	
	\emph{We handle now the case $1\leq r_0<r_1< p_1<p_0\leq\infty$.} Again, we are necessarily in the range of parameters \eqref{assumption:1}. Here, we denote
	\begin{equation*}
		\frac 1q=n\left(\frac1{r_1}-\frac 1{p_0}\right)
	\end{equation*}
	and observe that $q> 1$ holds by assumption.
	Therefore, employing \eqref{dispersion:3}, we find that
	\begin{equation*}
		\left\|\int_\mathbb{R}f(x-tv,v)\chi(t)dt\right\|_{L^{p_0}_xL^{r_0}_v}
		\leq \|\chi\|_{L^{q',p_1'}}
		\left\|f(x-tv,v)\right\|_{L^{q,p_1}_tL^{p_0}_xL^{r_0}_v}
		\lesssim
		\left\|f(x,v)\right\|_{L^{r_1}_xL^{p_1}_v},
	\end{equation*}
	which completes the proof for this case.
	
	It is to be emphasized here that the endpoint $q=1$ is not admitted in the preceding estimate because $\|\chi\|_{L^{\infty,\lambda}}=\infty$, for any $0<\lambda<\infty$, unless $\chi=0$.
	
	\emph{The remaining case $1\leq r_1<r_0\leq p_0<p_1\leq\infty$ follows by duality.} Indeed, observe that we are now in the range of parameters satisfying \eqref{assumption:2}. Then, we simply note that the adjoint operator to
	\begin{equation*}
		f(x,v)\mapsto\int_\mathbb{R}f(x-tv,v)\chi(t)dt
	\end{equation*}
	is given by
	\begin{equation*}
		g(x,v)\mapsto\int_\mathbb{R}g(x-tv,v)\chi(-t)dt.
	\end{equation*}
	Therefore, applying the estimates for the preceding cases to the adjoint operator, we infer the control
	\begin{equation*}
		\left\|\int_\mathbb{R}g(x-tv,v)\chi(-t)dt\right\|_{L^{r_1'}_xL^{p_1'}_v}
		\lesssim
		\|g(x,v)\|_{L^{p_0'}_xL^{r_0'}_v}.
	\end{equation*}
	The proof of the lemma is then concluded by duality.
\end{proof}

The next lemma is similar to the preceding result. It deals with the particular case of cutoffs $\chi(t)$ whose support does not contain the origin and offers a wider array of dispersive estimates under such an assumption. To be precise, notice that the constraint $\frac n{r_1}<1+\frac n{p_0}$ found in Lemma~\ref{dispersive:1} is removed from the statement below.

\begin{lem}\label{dispersive:2}
	Let $1\leq r_0\leq p_0\leq\infty$ and $1\leq r_1\leq p_1\leq\infty$ be such that
	\begin{equation*}
		\frac1{p_0}+\frac1{r_0}=\frac1{p_1}+\frac1{r_1}.
	\end{equation*}
	Further suppose that $(p_0,r_0)=(p_1,r_1)=(\infty,1)$ or
	\begin{equation}\label{assumption:3}
		\frac{n-1}{p_1}<\frac n{p_0}\leq \frac n{p_1},
	\end{equation}
	or
	\begin{equation}\label{assumption:4}
		\frac{n-1}{r_0'}<\frac n{r_1'}\leq \frac n{r_0'}.
	\end{equation}
	Observe that \eqref{assumption:3} and \eqref{assumption:4} cannot hold simultaneously unless $p_0=p_1$ and $r_0=r_1$.
	
	Then, for any compactly supported $\chi(t)\in L^\infty(\mathbb{R})$ such that the origin does not belong to the support of $\chi(t)$, one has the dispersive estimate
	\begin{equation*}
		\left\|\int_\mathbb{R}f(x-tv,v)\chi(t)dt\right\|_{L^{p_0}_xL^{r_0}_v}
		\leq C\left\|f(x,v)\right\|_{L^{r_1}_xL^{p_1}_v},
	\end{equation*}
	for some constant $C>0$ that only depends on $\chi$ and fixed parameters.
\end{lem}

\begin{proof}
	First of all, we observe that estimate \eqref{estimate:8} holds now for any $1\leq r_0\leq p_0\leq\infty$ without the restriction $n\left(\frac1{r_0}-\frac1{p_0}\right)<1$ because $\chi(t)$ is supported away from the origin. This settles the case $(p_0,r_0)=(p_1,r_1)$.
	
	Furthermore, the cases $1\leq r_0<r_1= p_1<p_0\leq\infty$ and $1\leq r_1<r_0= p_0<p_1\leq\infty$ have already been treated in \eqref{estimate:9} and a duality argument provided at the end of the proof of Lemma~\ref{dispersive:1}. In fact, the same duality argument shows that the case $1\leq r_1< r_0\leq p_0< p_1\leq\infty$ follows straightfowardly from the estimates in the range $1\leq r_0< r_1\leq p_1< p_0\leq\infty$.
	
	All in all, there only remains to handle the setting $1\leq r_0< r_1<p_1< p_0<\infty$ with $\frac{n-1}{p_1}<\frac n{p_0}$, which we establish now by an interpolation argument. To this end, we define
	\begin{equation*}
		0<\theta<1
		\qquad\text{and}\qquad
		1\leq \bar r< \bar a< \bar p< \infty
	\end{equation*}
	by
	\begin{equation*}
		\theta=\frac 1{r_1}-\frac 1{p_1},
		\qquad
		\bar r=\frac{(1-\theta)r_0}{1-\theta r_0},
		\qquad
		\bar a=(1-\theta)p_1,
		\qquad
		\bar p=(1-\theta)p_0.
	\end{equation*}
	One can then readily verify that
	\begin{equation*}
		\begin{aligned}
			\frac 1{r_0}&=\frac{1-\theta}{\bar r}+\frac\theta 1,
			&&&
			\frac 1{p_0}&=\frac{1-\theta}{\bar p}+\frac\theta \infty,
			\\
			\frac 1{r_1}&=\frac{1-\theta}{\bar a}+\frac\theta 1,
			&&&
			\frac 1{p_1}&=\frac{1-\theta}{\bar a}+\frac\theta \infty,
		\end{aligned}
	\end{equation*}
	and
	\begin{equation*}
		\frac1{\bar p}+\frac1{\bar r}=\frac2{\bar a}
		\qquad\text{and}\qquad
		\frac{n-1}{\bar a}<\frac n{\bar p}.
	\end{equation*}
	In particular, an application of \eqref{estimate:9} yields that
	\begin{equation}\label{estimate:10}
		\left\|\int_\mathbb{R}f(x-tv,v)\chi(t)dt\right\|_{L^{\bar p}_xL^{\bar r}_v}
		\lesssim
		\left\|f(x,v)\right\|_{L^{\bar a}_{x,v}},
	\end{equation}
	which we are now about to interpolate with estimate \eqref{estimate:8} in the form
	\begin{equation}\label{estimate:11}
		\left\|\int_\mathbb{R}f(x-tv,v)\chi(t)dt\right\|_{L^{\infty}_xL^{1}_v}
		\lesssim
		\left\|f(x,v)\right\|_{L^{1}_xL^{\infty}_v}.
	\end{equation}
	In particular, recall that standard results from the theory of complex interpolation of Lebesgue spaces (e.g., see \cite[Theorem 2.2.6]{hvvl16}) show that
	\begin{equation*}
		\big(L^{\bar p_0}_xL^{\bar r_0}_v,L^\infty_xL^1_v\big)_{[\theta]}=L^{p_0}_xL^{r_0}_v
		\quad\text{and}\quad
		\big(L^{\bar a}_{x,v},L^1_xL^\infty_v\big)_{[\theta]}=L^{r_1}_xL^{p_1}_v.
	\end{equation*}
	Therefore, interpolating estimates \eqref{estimate:10} and \eqref{estimate:11} (these estimates remain valid for complex valued functions), we finally deduce that
	\begin{equation*}
		\left\|\int_\mathbb{R}f(x-tv,v)\chi(t)dt\right\|_{L^{p_0}_xL^{r_0}_v}
		\lesssim
		\left\|f(x,v)\right\|_{L^{r_1}_xL^{p_1}_v},
	\end{equation*}
	thereby completing the proof of the lemma.
\end{proof}

\begin{rema}
	The preceding two lemmas offer a spectrum of dispersive estimates wider than the results from \cite[Section~2]{am19}. They can be used to extend some results from \cite{am19}, namely Propositions 3.4 and 3.5 and Theorem 3.6 therein, to a larger range of parameters.
\end{rema}

At last, we combine the dispersive estimates from the preceding lemma with the energy method to produce our final theorem, which can be interpreted as a fusion of Theorem \ref{result:2} with Theorem~3.6 from \cite{am19}.

This new result generalizes greatly the duality principle provided by Theorem \ref{result:2}. Indeed, it shows that one can trade integrability between the particle density $f$ and the source term $v\cdot\nabla_x f$, at the same time as one transfers integrability between the variables $x$ and $v$, without affecting the resulting smoothness of velocity averages. The general principle below requires that the harmonic mean of all integrability parameters be equal to $2$. However, some additional restrictions on the range of applicable parameters are needed. These constraints remain natural, though, but the proof of their optimality would require suitable counterexamples.

\begin{thm}\label{result:8}
	Let $f(x,v)$ be compactly supported (in all variables) and such that
	\begin{equation*}
		\begin{aligned}
			f & \in L^{r_0}(\mathbb{R}^n_x;L^{p_0}(\mathbb{R}^n_v)),
			\\
			v\cdot\nabla_x f & \in L^{r_1}(\mathbb{R}^n_x;L^{p_1}(\mathbb{R}^n_v))\cap L^{r_2}(\mathbb{R}^n_x;L^{r_2'}(\mathbb{R}^n_v)),
		\end{aligned}
	\end{equation*}
	for some given $1< r_0\leq p_0\leq \infty$, $1<r_1\leq p_1<\infty$ and $\frac{2n}{n+1}<r_2\leq 2$  such that
	\begin{equation*}
		2=\frac1{r_0}+\frac1{p_0}+\frac1{r_1}+\frac1{p_1}
		\qquad
		\text{and}
		\qquad
		\frac{n-1}{p_0}<\frac n{r_1'}<\frac n{p_0}+\frac 1{p_1}.
	\end{equation*}
	Then $\tilde f$ belongs to $\dot H^\frac 12(\mathbb{R}^n_x)$.
	
	More precisely, for any compact set $K\subset\mathbb{R}^n_v$, there exists $C_{K}>0$ (which also depends on all integrability parameters) such that, for all compactly supported $f(x,v)$ vanishing for $v$ outside $K$,
	\begin{equation*}
		\big\|\tilde f\big\|_{\dot H^\frac12_x}^2
		\leq C_K \|f\|_{L^{r_0}_xL^{p_0}_v}
		\|v\cdot\nabla_xf\|_{L^{r_1}_xL^{p_1}_v}
		+C_K\|v\cdot\nabla_xf\|_{L^{r_2}_xL^{r_2'}_v}^2.
	\end{equation*}
	In particular, the constant $C_{K}$ is independent of the size of the support of $f(x,v)$ in $x$. Moreover, whenever $\frac{p_1}{r_1}< \frac{p_0}{r_0}+\frac{p_1}n$, the above inequality remains true for functions that are compactly supported in $v$ but not necessarily in $x$.
\end{thm}

\begin{proof}
	Thanks to Lemma \ref{approximation:5} from the appendix, observe that we only need to consider smooth compactly supported functions. Furthermore, noticing that $\frac{r_2'}{r_2}< 1+\frac{r_2'}n\leq \frac{p_0}{r_0}+\frac{r_2'}n$ because $\frac{2n}{n+1}<r_2\leq 2$, it is readily seen, if one additionally requires $\frac{p_1}{r_1}< \frac{p_0}{r_0}+\frac{p_1}n$, that the constraint \eqref{constraint:1} holds, which implies that the theorem extends to functions that are not necessarily compactly supported in $x$ by Lemma \ref{approximation:3}. Either way, we only need now to consider functions $f(x,v)\in C_c^\infty(\mathbb{R}^n\times\mathbb{R}^n)$.
	
	Following the energy method developed in Section \ref{section:2}, our starting point here is the estimate \eqref{energy:5}, which we then combine with the parametrix representation formula \eqref{parametrix:1} where $\chi_0(t)$ is chosen so that it is supported away from the origin. This leads to
	\begin{equation*}
		\begin{aligned}
			\big\||D_x|^\frac 12\tilde f\big\|_{L^2_x}^2
			&\leq C_K\operatorname{Re} \langle v\cdot\nabla_xf,M(D_x,D_v)f\rangle_{L^2_{x,v}}
			\\
			&= C_K
			\operatorname{Re} \Big\langle v\cdot\nabla_xf,M(D_x,D_v)\int_\mathbb{R}f(x-tv,v)\chi_0(t)dt\Big\rangle_{L^2_{x,v}}
			\\
			&\quad + C_K
			\operatorname{Re} \Big\langle v\cdot\nabla_xf,M(D_x,D_v)\int_\mathbb{R}v\cdot\nabla_xf(x-tv,v)\chi_1(t)dt\Big\rangle_{L^2_{x,v}},
		\end{aligned}
	\end{equation*}
	where $C_K=2\pi \sup_{j=1,\ldots,n}|\pi_{n-1}^jK|_{n-1}$ only depends on the compact set $K$ containing the $v$-support of $f$.
	
	Further recall from Section \ref{section:2} that the Fourier multiplier operator $M(D_x,D_v)$ defined by \eqref{multiplier:1} is bounded over any reflexive mixed Lebesgue space. It therefore follows that
	\begin{equation*}
		\begin{aligned}
			\big\|\tilde f\big\|_{\dot H^\frac12_x}^2
			& \lesssim \|v\cdot\nabla_xf\|_{L^{r_1}_xL^{p_1}_v}
			\left\|\int_\mathbb{R}f(x-tv,v)\chi_0(t)dt\right\|_{L^{r_1'}_xL^{p_1'}_v}
			\\
			&\quad +\|v\cdot\nabla_xf\|_{L^{r_2}_xL^{r_2'}_v}
			\left\|\int_\mathbb{R}v\cdot\nabla_x f(x-tv,v)\chi_1(t)dt\right\|_{L^{r_2'}_xL^{r_2}_v}.
		\end{aligned}
	\end{equation*}
	Finally, we observe that the integrability parameters satisfy
	\begin{equation*}
		\frac{n-1}{p_0}<\frac n{r_1'}\leq\frac n{p_0}
		\qquad\text{or}\qquad
		\frac {n}{p_0}\leq \frac n{r_1'}<\frac n{p_0}+\frac 1{p_1},
	\end{equation*}
	which is equivalent to
	\begin{equation*}
		\frac{n-1}{p_0}<\frac n{r_1'}\leq\frac n{p_0}
		\qquad\text{or}\qquad
		\frac {n-1}{p_1}< \frac n{r_0'}\leq\frac n{p_1},
	\end{equation*}
	respectively. This allows us to apply the dispersive inequalities from Lemmas \ref{dispersive:1} and \ref{dispersive:2} (recall that $\chi_0(t)$ is supported away from the origin) to the preceding estimate to deduce that
	\begin{equation*}
		\big\|\tilde f\big\|_{\dot H^\frac12_x}^2
		\lesssim \|v\cdot\nabla_xf\|_{L^{r_1}_xL^{p_1}_v}
		\|f\|_{L^{r_0}_xL^{p_0}_v}
		+\|v\cdot\nabla_xf\|_{L^{r_2}_xL^{r_2'}_v}^2,
	\end{equation*}
	which concludes the proof of the theorem.
\end{proof}

\begin{rema}
	While we have previously insisted on the fact that the contributions from this article sometimes intersect with the work from \cite{jlt20}, we would like to point out that the content of the present section is disjoint from the results from \cite{jlt20}, where the effects of dispersion on the energy method were not considered.
\end{rema}

The next two corollaries are mere restatements of the previous theorem where we set $r_2=r_1$ and exploit the fact that $f(x,v)$ is compactly supported in $v$. Their justification based on Theorem \ref{result:8} is elementary.

\begin{cor}\label{result:9}
	Let $f(x,v)$ be compactly supported (in all variables) and such that
	\begin{equation*}
		\begin{aligned}
			f & \in L^{r_0}(\mathbb{R}^n_x;L^{p_0}(\mathbb{R}^n_v)),
			\\
			v\cdot\nabla_x f & \in L^{r_1}(\mathbb{R}^n_x;L^{p_1}(\mathbb{R}^n_v)),
		\end{aligned}
	\end{equation*}
	for some given $1< r_0\leq p_0\leq \infty$ and $\frac{2n}{n+1}<r_1\leq 2\leq r_1'\leq p_1<\infty$  such that
	\begin{equation*}
		2=\frac1{r_0}+\frac1{p_0}+\frac1{r_1}+\frac1{p_1}
		\qquad
		\text{and}
		\qquad
		\frac{n-1}{p_0}<\frac n{r_1'}<\frac n{p_0}+\frac 1{p_1}.
	\end{equation*}
	Then $\tilde f$ belongs to $\dot H^\frac 12(\mathbb{R}^n_x)$.
	
	Moreover, whenever $\frac{p_1}{r_1}< \frac{p_0}{r_0}+\frac{p_1}n$, the result remains true for functions that are compactly supported in $v$ but not necessarily in $x$.
\end{cor}

\begin{proof}
	This result follows directly from Theorem \ref{result:8} by setting $r_2=r_1$ therein and observing that
	\begin{equation*}
		L^{r_1}(\mathbb{R}^n_x;L^{p_1}(K))
		\subset
		L^{r_1}(\mathbb{R}^n_x;L^{r_1'}(K)),
	\end{equation*}
	for any compact set $K\subset \mathbb{R}^n_v$, because $p_1\geq r_1'$.
\end{proof}

\begin{cor}\label{result:10}
	Let $f(x,v)$ be compactly supported in $v$ and such that
	\begin{equation*}
		\begin{aligned}
			f & \in L^{r_0}(\mathbb{R}^n_x;L^{p_0}(\mathbb{R}^n_v)),
			\\
			v\cdot\nabla_x f & \in L^{r_1}(\mathbb{R}^n_x;L^{r_1'}(\mathbb{R}^n_v)),
		\end{aligned}
	\end{equation*}
	for some given $1< r_0\leq p_0\leq \infty$ and $\frac{2n}{n+1}<r_1\leq 2$  such that
	\begin{equation*}
		\frac 2{r_1'}\leq \frac 1{r_0}+\frac 1{p_0}\leq 1
		\qquad
		\text{and}
		\qquad
		\frac {n-1}{r_1'}<\frac {n-1}{p_0}+\frac 1{r_0'}.
	\end{equation*}
	Then $\tilde f$ belongs to $\dot H^\frac 12(\mathbb{R}^n_x)$.
\end{cor}

\begin{proof}
	We define $p_1$ by
	\begin{equation*}
		\frac 1{p_1}=2-\frac 1{r_0}-\frac 1{p_0}-\frac 1{r_1}.
	\end{equation*}
	It is then easy to verify that
	\begin{equation*}
		r_1\leq p_1\leq r_1',
		\qquad
		\frac{n-1}{p_0}\leq\frac{n-1}2<\frac n{r_1'}
		\qquad
		\text{and}
		\qquad
		\frac n{r_1'}<\frac n{p_0}+\frac 1{p_1}.
	\end{equation*}
	In particular, using that $f$ is compactly supported in $v$, we infer that
	\begin{equation*}
		v\cdot\nabla_x f \in L^{r_1}(\mathbb{R}^n_x;L^{p_1}(\mathbb{R}^n_v)).
	\end{equation*}
	One may therefore apply Theorem \ref{result:8}, with $r_2=r_1$, to deduce that $\tilde f$ belongs to $\dot H^\frac 12$.
	
	Finally, since
	\begin{equation*}
		\frac{p_1}{r_1}<1+\frac{p_1}n\leq \frac{p_0}{r_0}+\frac{p_1}n
	\end{equation*}
	because $\frac{2n}{n+1}<r_1\leq p_1< \frac{2n}{n-1}$ and $r_0\leq p_0$,
	we conclude that $f(x,v)$ does not need to be compactly supported in $x$ in order to apply Theorem \ref{result:8}, which completes the justification of the corollary.
\end{proof}

In conclusion, this article establishes a considerable amount of new velocity averaging lemmas leading to a maximal smoothness of the averages in $\dot H^{1/2}$. It seems that the common underlying principle dictating the maximal regularity of averages consists in setting the harmonic mean of all integrability parameters at the value $2$. More work is now needed to understand the precise limits of applicability of such principles.

Note also that some interesting research directions are left open-ended here. For instance, one may naturally wonder now how the dispersive methods from the present section would fare in the settings from Section \ref{section:3}, which deal with singular sources. The combination of dispersive and hypoelliptic techniques in the energy method would likely come with some technical but interesting challenges, whose precise outcome remains unclear to us at the moment.

Finally, as previously mentioned in the introduction, recall that velocity averaging lemmas have important consequences on the regularity theory of non-linear systems, when kinetic formulations are available. This is best exemplified by the results from \cite{jp02} concerning multidimensional scalar conservation laws, one-dimensional isentropic gas dynamics and Ginzburg--Landau models in micromagnetics. We are hopeful that the results and methods from this article will find similar applications to kinetic formulations, thereby improving and expanding the current regularity theories for kinetic equations and hyperbolic systems.

\appendix

\section{Approximation procedures in kinetic transport equations}
\label{approximation:1}

We explain here how a standard use of Friedrichs' commutation method combined with renormalization techniques can be employed to build approximation procedures for the stationary kinetic transport equation \eqref{transport:1}. In particular, the following collection of lemmas shows that it is sufficient to establish velocity averaging estimates for smooth compactly supported functions in order to deduce more general results by density arguments.

\begin{lem}\label{approximation:2}
	Consider any set of parameters $1< p_0,p_1,q_0,q_1,r<\infty$ and $s\geq 0$. Let us assume that, for any compact set $K\subset\mathbb{R}^n$, there exists a finite constant $C_K>0$ such that the estimate
	\begin{equation}\label{general:1}
		\left\|\int_{K}fdv\right\|_{\dot W^{s,r}(\mathbb{R}^n_x)}\leq C_K
		\Big(\|f\|_{L^{p_0}(\mathbb{R}^n_x;L^{q_0}(\mathbb{R}^n_v))}
		+\|v\cdot\nabla_xf\|_{L^{p_1}(\mathbb{R}^n_x;L^{q_1}(\mathbb{R}^n_v))}\Big)
	\end{equation}
	holds for any $f(x,v)\in C_c^\infty(\mathbb{R}^n\times\mathbb{R}^n)$. Then, the above estimate remains valid for any $f(x,v)$ such that
	\begin{equation}\label{general:2}
		\begin{aligned}
			f&\in L^{p_0}(\mathbb{R}^n_x;L^{q_0}(\mathbb{R}^n_v))\cap L^{p_1}(\mathbb{R}^n_x;L^{q_1}(\mathbb{R}^n_v)),
			\\
			v\cdot\nabla_x f&\in L^{p_1}(\mathbb{R}^n_x;L^{q_1}(\mathbb{R}^n_v)).
		\end{aligned}
	\end{equation}
\end{lem}

\begin{rema}
	The results in this appendix also apply to variants of the velocity averaging estimate \eqref{general:1}. For instance, similar results hold if one considers the inequality
	\begin{equation}\label{general:3}
		\left\|\int_{K}fdv\right\|_{\dot W^{s,r}(\mathbb{R}^n_x)}\leq C_K
		\|f\|_{L^{p_0}(\mathbb{R}^n_x;L^{q_0}(\mathbb{R}^n_v))}^{1-\theta}
		\|v\cdot\nabla_xf\|_{L^{p_1}(\mathbb{R}^n_x;L^{q_1}(\mathbb{R}^n_v))}^\theta,
	\end{equation}
	for some given $0<\theta<1$, instead of \eqref{general:1}.
\end{rema}

\begin{proof}
	Let us consider any $f(x,v)$ satisfying the bounds \eqref{general:2}. Then, taking some cutoff functions $\phi(x,v),\chi(x,v)\in C_c^\infty(\mathbb{R}^n\times\mathbb{R}^n)$ and defining
	\begin{equation*}
		\phi_{\epsilon}(x,v)=\epsilon^{-2n}\phi(x/\epsilon,v/\epsilon),
		\qquad
		\chi_{\epsilon}(x,v)=\chi(\epsilon x,\epsilon v),
	\end{equation*}
	with $0<\epsilon\leq 1$, a direct commutator estimate (akin to Friedrichs' commutation lemma; see \cite[Lemma 17.1.5]{h07}, for instance) yields that
	\begin{equation*}
		\begin{aligned}
			\|\phi_{\epsilon}*(\chi_{\epsilon} v\cdot\nabla_x f)
			-v\cdot\nabla_x(\phi_{\epsilon}* (\chi_{\epsilon} f))\|_{L^{p_1}_xL^{q_1}_v}\hspace{-60mm}&
			\\
			&\leq
			\|\phi_{\epsilon}*(v\cdot\nabla_x (\chi_{\epsilon} f))
			-v\cdot\nabla_x(\phi_{\epsilon}* (\chi_{\epsilon} f))\|_{L^{p_1}_xL^{q_1}_v}
			+\|\phi_{\epsilon}*(f v\cdot\nabla_x \chi_{\epsilon})
			\|_{L^{p_1}_xL^{q_1}_v}
			\\
			&=
			\|(v\cdot\nabla_x \phi_{\epsilon})*(\chi_{\epsilon}f)\|_{L^{p_1}_xL^{q_1}_v}
			+\|\phi_{\epsilon}*(f v\cdot\nabla_x \chi_{\epsilon})
			\|_{L^{p_1}_xL^{q_1}_v}
			\\
			&\leq
			\left(
			\|v\cdot\nabla_x\phi\|_{L^1_{x,v}}\|\chi\|_{L^\infty_{x,v}}
			+\|\phi\|_{L^1_{x,v}}\|v\cdot\nabla_x\chi\|_{L^\infty_{x,v}}\right)
			\|f\|_{L^{p_1}_xL^{q_1}_v}.
		\end{aligned}
	\end{equation*}
	In particular, since $C_c^\infty$ is dense in $L^{p_1}L^{q_1}$, further assuming $\int_{\mathbb{R}^{2n}}\phi(x,v) dxdv=1$ and $\chi(0,0)=1$, we obtain that
	\begin{equation*}
		\lim_{\epsilon\to 0}
		\|\phi_\epsilon*(\chi_\epsilon v\cdot\nabla_x f)-v\cdot\nabla_x(\phi_\epsilon* (\chi_\epsilon f))\|_{L^{p_1}_xL^{q_1}_v}
		=0,
	\end{equation*}
	whence
	\begin{equation*}
		\lim_{\epsilon\to 0}
		\|v\cdot\nabla_x(\phi_\epsilon* (\chi_\epsilon f))\|_{L^{p_1}_xL^{q_1}_v}
		=\|v\cdot\nabla_xf\|_{L^{p_1}_xL^{q_1}_v}.
	\end{equation*}
	
	Now, since $f_\epsilon=\phi_\epsilon* (\chi_\epsilon f)$ is compactly supported and smooth,
	one deduces from \eqref{general:1},
	for any $\psi(x)\in C_c^\infty(\mathbb{R}^n)$, that
	\begin{equation*}
		\begin{aligned}
			\left|\int_{\mathbb{R}^n}{\psi}|D_x|^s\int_Kfdvdx\right|
			& \leq
			\left|\int_{\mathbb{R}^n}{|D_x|^s\psi}
			\int_{K}(f-f_\epsilon)dvdx\right|
			+
			\left|\int_{\mathbb{R}^n}{\psi}|D_x|^s\int_{K}f_\epsilon dvdx\right|
			\\
			&\leq
			\big\||D_x|^s\psi\big\|_{L^{p_0'}_x}
			|K|^\frac 1{q_0'}
			\|f-f_\epsilon\|_{L^{p_0}_xL^{q_0}_v}
			\\
			&\quad+C_{K}\|\psi\|_{L^{r'}_x}
			\Big(\|f_\epsilon\|_{L^{p_0}_xL^{q_0}_v}
			+\|v\cdot\nabla_xf_\epsilon\|_{L^{p_1}_xL^{q_1}_v}\Big),
		\end{aligned}
	\end{equation*}
	whereby, letting $\epsilon$ tend to zero,
	\begin{equation*}
		\begin{aligned}
			\left|\int_{\mathbb{R}^n}{\psi}|D_x|^s\int_Kfdvdx\right|
			\leq
			C_{K}\|\psi\|_{L^{r'}_x}
			\Big(\|f\|_{L^{p_0}_xL^{q_0}_v}
			+\|v\cdot\nabla_xf\|_{L^{p_1}_xL^{q_1}_v}\Big).
		\end{aligned}
	\end{equation*}
	Finally, taking the supremum over all $\psi\in L^{r'}$ in the previous estimate establishes that \eqref{general:1} holds for every $f$ satisfying \eqref{general:2}.
\end{proof}

Using renormalization methods in conjunction with the previous lemma, we obtain the following result.

\begin{lem}\label{approximation:3}
	Consider any set of parameters $1< p_0,p_1,q_0,q_1,r<\infty$ and $s\geq 0$ such that
	\begin{equation}\label{constraint:1}
		\frac{q_1}{p_1}< \frac{q_0}{p_0}+\frac{q_1}n.
	\end{equation}
	Let us assume that, for any compact set $K\subset\mathbb{R}^n$, there exists a finite constant $C_K>0$ such that estimate \eqref{general:1}, or \eqref{general:3},
	holds for any $f(x,v)\in C_c^\infty(\mathbb{R}^n\times\mathbb{R}^n)$. Then \eqref{general:1}, respectively \eqref{general:3}, remains valid for any $f(x,v)$ such that
	\begin{equation*}
		\begin{aligned}
			f&\in L^{p_0}(\mathbb{R}^n_x;L^{q_0}(\mathbb{R}^n_v)),
			\\
			v\cdot\nabla_x f&\in L^{p_1}(\mathbb{R}^n_x;L^{q_1}(\mathbb{R}^n_v)).
		\end{aligned}
	\end{equation*}
\end{lem}

\begin{proof}
	The justification of this result relies on a combination of Lemma \ref{approximation:2} with renormalization techniques.
	
	\emph{We begin by assuming that $\frac{q_1}{p_1}\leq \frac{q_0}{p_0}$.} In this case, we introduce a cutoff function $\rho\in C^1(\mathbb{R})$ such that
	\begin{equation*}
		\mathds{1}_{\{|s|\geq 2\}}\leq \rho(s)\leq\mathds{1}_{\{|s|\geq 1\}}
		\quad\text{and}\quad
		|\rho'(s)|\leq 2
	\end{equation*}
	and we define the renormalizations
	\begin{equation}\label{renormalization:1}
		h_\lambda(x,v)=\frac{f(x,v)}{(1+\lambda^2 f(x,v)^2)^\frac 12}\rho\left(\lambda^{-1}f(x,v)\right)\mathds{1}_K(v),
	\end{equation}
	where $0<\lambda\leq 1$. In particular, for any $\alpha\geq 0$, one can show that
	\begin{equation}\label{renormalization:3}
		|h_\lambda|\leq \lambda^{-|\alpha-1|}|f|^{\alpha}\mathds{1}_K(v)
	\end{equation}
	and
	\begin{equation*}
		\begin{aligned}
			|v\cdot\nabla_x h_\lambda|&= |v\cdot\nabla_x f|
			\left|\frac {\rho'\left(f/\lambda\right)f/\lambda}{(1+\lambda^2f^2)^\frac 12}
			+\frac {\rho\left(f/\lambda\right)}{(1+\lambda^2f^2)^\frac 32}\right|\mathds{1}_K(v)
			\\
			&\leq 5|v\cdot\nabla_x f|\mathds{1}_K(v).
		\end{aligned}
	\end{equation*}
	It therefore follows that
	\begin{equation}\label{renormalization:2}
		\lim_{\lambda\to 0}\|h_\lambda\|_{L_x^{p_0}L_v^{q_0}} = \|f\mathds{1}_K(v)\|_{L_x^{p_0}L_v^{q_0}},
		\qquad
		\lim_{\lambda\to 0}\|v\cdot\nabla_x h_\lambda\|_{L_x^{p_1}L_v^{q_1}} = \|v\cdot\nabla_x f\mathds{1}_K(v)\|_{L_x^{p_1}L_v^{q_1}},
	\end{equation}
	and, recalling that $\frac{q_1p_0}{p_1}\leq q_0$,
	\begin{equation*}
		\begin{aligned}
			\|h_\lambda\|_{L_x^{p_1}L_v^{q_1}}
			& \leq \lambda^{-\left|\frac{p_0}{p_1}-1\right|}\left\||f|^{p_0/p_1}\mathds{1}_K(v)\right\|_{L_x^{p_1}L_v^{q_1}}
			\\ & =\lambda^{-\left|\frac{p_0}{p_1}-1\right|}\|f\mathds{1}_K(v)\|_{L_x^{p_0}L_v^{q_1p_0/p_1}}^{p_0/p_1}
			\\ & \leq C_{\lambda,K}\|f\|_{L_x^{p_0}L_v^{q_0}}^{p_0/p_1}<\infty.
		\end{aligned}
	\end{equation*}
	
	Since the renormalizations $h_\lambda$ satisfy the bounds \eqref{general:2}, we may now employ Lemma \ref{approximation:2} to deduce from \eqref{general:1} that
	\begin{equation*}
		\begin{aligned}
			\left\|\int_{K}fdv\right\|_{\dot W^{s,r}(\mathbb{R}^n_x)}
			&\leq \liminf_{\lambda\to 0}\left\|\int_{K}h_\lambda dv\right\|_{\dot W^{s,r}(\mathbb{R}^n_x)}
			\\
			&\leq C_K\lim_{\lambda\to 0}
			\Big(\|h_\lambda\|_{L^{p_0}(\mathbb{R}^n_x;L^{q_0}(\mathbb{R}^n_v))}
			+\|v\cdot\nabla_xh_\lambda\|_{L^{p_1}(\mathbb{R}^n_x;L^{q_1}(\mathbb{R}^n_v))}\Big)
			\\
			&= C_K
			\Big(\|f\|_{L^{p_0}(\mathbb{R}^n_x;L^{q_0}(\mathbb{R}^n_v))}
			+\|v\cdot\nabla_xf\|_{L^{p_1}(\mathbb{R}^n_x;L^{q_1}(\mathbb{R}^n_v))}\Big),
		\end{aligned}
	\end{equation*}
	which concludes the proof of the lemma when $\frac{q_1}{p_1}\leq \frac{q_0}{p_0}$.
	
	\emph{Next, we consider the case $\frac{q_0}{p_0}\leq \frac{q_1}{p_1}\leq \frac{q_0}{p_0}+\frac{q_1}n$.} In order to handle this setting, we need to further localize in the variable $x$. More precisely, for any $0<\epsilon,\lambda\leq 1$, we introduce now
	\begin{equation*}
		H_{\epsilon,\lambda}(x,v)=\chi(\epsilon x)h_\lambda(x,v),
	\end{equation*}
	where the renormalization $h_\lambda$ is defined in \eqref{renormalization:1} and the cutoff $\chi\in C_c^\infty(\mathbb{R}^n)$ satisfies
	\begin{equation*}
		\mathds{1}_{\{|x|\leq 1\}}\leq \chi(x)\leq\mathds{1}_{\{|x|\leq 2\}}.
	\end{equation*}
	Then, in view of \eqref{renormalization:2}, one has that
	\begin{equation}\label{renormalization:4}
		\lim_{\epsilon,\lambda\to 0}\|H_{\epsilon,\lambda}\|_{L_x^{p_0}L_v^{q_0}} = \|f\mathds{1}_K(v)\|_{L_x^{p_0}L_v^{q_0}}.
	\end{equation}
	
	Now, employing \eqref{renormalization:3}, we obtain that
	\begin{equation*}
		\begin{aligned}
			\|v\cdot\nabla_x H_{\epsilon,\lambda}\|_{L_x^{p_1}L_v^{q_1}}
			& \leq \|v\cdot\nabla_x h_\lambda\|_{L_x^{p_1}L_v^{q_1}}
			+\epsilon\|(v\cdot\nabla\chi)(\epsilon x) h_{\lambda}\|_{L_x^{p_1}L_v^{q_1}}
			\\
			& \leq \|v\cdot\nabla_x h_\lambda\|_{L_x^{p_1}L_v^{q_1}}
			+\lambda^{-|q_0-q_1|/q_1}\epsilon
			\left\||v|\mathds{1}_{\{1\leq|\epsilon x|\leq 2\}} f^{q_0/q_1}\mathds{1}_K(v)\right\|_{L_x^{p_1}L_v^{q_1}}.
		\end{aligned}
	\end{equation*}
	Further setting $\frac 1a=\frac{q_1}{p_1q_0}-\frac 1{p_0}\leq 0$, we deduce, by H\"older's inequality,
	\begin{equation*}
		\begin{aligned}
			\left\|\mathds{1}_{\{1\leq|\epsilon x|\leq 2\}} f^{q_0/q_1}\right\|_{L_x^{p_1}L_v^{q_1}}
			&=\left\|\mathds{1}_{\{1\leq|\epsilon x|\leq 2\}} f\right\|_{L_x^{p_1q_0/q_1}L_v^{q_0}}^{q_0/q_1}
			\\
			&\leq \left\|\mathds{1}_{\{1\leq|\epsilon x|\leq 2\}}\right\|_{L_x^a}^{q_0/q_1}
			\left\|f\right\|_{L_x^{p_0}L_v^{q_0}}^{q_0/q_1}
			\leq C\epsilon^{-\frac{nq_0}{aq_1}}\left\|f\right\|_{L_x^{p_0}L_v^{q_0}}^{q_0/q_1},
		\end{aligned}
	\end{equation*}
	where $C>0$ denotes a constant depending only on fixed parameters. Therefore, combining the preceding estimates, we arrive at the control
	\begin{equation*}
		\|v\cdot\nabla_x H_{\epsilon,\lambda}\|_{L_x^{p_1}L_v^{q_1}}
		\leq
		\|v\cdot\nabla_x h_\lambda\|_{L_x^{p_1}L_v^{q_1}}
		+C\lambda^{-\frac{|q_0-q_1|}{q_1}}\epsilon^{1-\frac{nq_0}{aq_1}}
		\left\|f\right\|_{L_x^{p_0}L_v^{q_0}}^{q_0/q_1}.
	\end{equation*}
	Observe that $\frac{nq_0}{aq_1}<1$ by virtue of \eqref{constraint:1}. Thus, we can assume from now on that the parameters $\epsilon$ and $\lambda$ are chosen ($\lambda=|\log\epsilon|^{-1}$ will do) so that
	\begin{equation*}
		\lim_{\epsilon,\lambda\to 0}\lambda^{-\frac{|q_0-q_1|}{q_1}}\epsilon^{1-\frac{nq_0}{aq_1}}=0.
	\end{equation*}
	Consequently, we infer, employing \eqref{renormalization:2}, that
	\begin{equation}\label{renormalization:5}
		\limsup_{\epsilon,\lambda\to 0}\|v\cdot\nabla_x H_{\epsilon,\lambda}\|_{L_x^{p_1}L_v^{q_1}}
		\leq
		\|v\cdot\nabla_x f\|_{L_x^{p_1}L_v^{q_1}}.
	\end{equation}
	
	All in all, recalling that $H_{\epsilon,\lambda}$ is compactly supported and bounded pointwise, an application of Lemma \ref{approximation:2} combined with \eqref{renormalization:4} and \eqref{renormalization:5} allows us to conclude that
	\begin{equation*}
		\begin{aligned}
			\left\|\int_{K}fdv\right\|_{\dot W^{s,r}(\mathbb{R}^n_x)}
			&\leq \liminf_{\epsilon,\lambda\to 0}\left\|\int_{K}H_{\epsilon,\lambda} dv\right\|_{\dot W^{s,r}(\mathbb{R}^n_x)}
			\\
			&\leq C_K\liminf_{\epsilon,\lambda\to 0}
			\Big(\|H_{\epsilon,\lambda}\|_{L^{p_0}(\mathbb{R}^n_x;L^{q_0}(\mathbb{R}^n_v))}
			+\|v\cdot\nabla_xH_{\epsilon,\lambda}\|_{L^{p_1}(\mathbb{R}^n_x;L^{q_1}(\mathbb{R}^n_v))}\Big)
			\\
			&= C_K
			\Big(\|f\|_{L^{p_0}(\mathbb{R}^n_x;L^{q_0}(\mathbb{R}^n_v))}
			+\|v\cdot\nabla_xf\|_{L^{p_1}(\mathbb{R}^n_x;L^{q_1}(\mathbb{R}^n_v))}\Big),
		\end{aligned}
	\end{equation*}
	which completes the proof of the lemma.
\end{proof}

An alternative renormalization and localization approach provides the following variation on Lemma \ref{approximation:3}.

\begin{lem}\label{approximation:4}
	Consider any set of parameters $1< p_0,p_1,q_0,q_1,r<\infty$ and $s\geq 0$ such that
	\begin{equation}\label{constraint:2}
		\frac{q_1}{p_1}< \frac{n-1}{n}+\frac{q_1}n.
	\end{equation}
	Let us assume that, for any compact set $K\subset\mathbb{R}^n$, there exists a finite constant $C_K>0$ such that estimate \eqref{general:1}, or \eqref{general:3},
	holds for any $f(x,v)\in C_c^\infty(\mathbb{R}^n\times\mathbb{R}^n)$. Then \eqref{general:1}, respectively \eqref{general:3}, remains valid for any $f(x,v)$ such that
	\begin{equation*}
		\begin{aligned}
			f&\in L^{p_0}(\mathbb{R}^n_x;L^{q_0}(\mathbb{R}^n_v)),
			\\
			v\cdot\nabla_x f&\in L^{p_1}(\mathbb{R}^n_x;L^{q_1}(\mathbb{R}^n_v)).
		\end{aligned}
	\end{equation*}
\end{lem}

\begin{proof}
	We follow a strategy similar to the proof of Lemma \ref{approximation:3}.
	
	For any $0<\epsilon,\delta,\lambda\leq 1$, let us introduce the renormalization
	\begin{equation*}
		H_{\epsilon,\delta,\lambda}(x,v)=\chi(\epsilon x)\chi(\delta\pi_v x)\frac{f(x,v)}{(1+\lambda^2 f(x,v)^2)^\frac 12}\mathds{1}_K(v),
	\end{equation*}
	where the cutoff $\chi\in C_c^\infty(\mathbb{R}^n)$ satisfies
	\begin{equation*}
		\mathds{1}_{\{|x|\leq 1\}}\leq \chi(x)\leq\mathds{1}_{\{|x|\leq 2\}}
	\end{equation*}
	and the orthogonal projection $\pi_v$ is given by
	\begin{equation*}
		\pi_v x=
		x-\left(x\cdot\frac{v}{|v|}\right)\frac{v}{|v|}.
	\end{equation*}
	In particular, observe that $H_{\epsilon,\delta,\lambda}(x,v)$ is compactly supported and bounded pointwise. Therefore, we deduce from Lemma~\ref{approximation:2} that estimate \eqref{general:1} holds for $H_{\epsilon,\delta,\lambda}$, whereby
	\begin{equation*}
		\begin{aligned}
			\left\|\int_{K}fdv\right\|_{\dot W^{s,r}(\mathbb{R}^n_x)}
			&\leq \liminf_{\epsilon,\delta,\lambda\to 0}\left\|\int_{K}H_{\epsilon,\delta,\lambda} dv\right\|_{\dot W^{s,r}(\mathbb{R}^n_x)}
			\\
			&\leq C_K\liminf_{\epsilon,\delta,\lambda\to 0}
			\Big(\|H_{\epsilon,\delta,\lambda}\|_{L^{p_0}(\mathbb{R}^n_x;L^{q_0}(\mathbb{R}^n_v))}
			+\|v\cdot\nabla_xH_{\epsilon,\delta,\lambda}\|_{L^{p_1}(\mathbb{R}^n_x;L^{q_1}(\mathbb{R}^n_v))}\Big).
		\end{aligned}
	\end{equation*}
	
	Now, since $C_c^\infty$ is dense in $L^{p_0}L^{q_0}$ and $L^{p_1}L^{q_1}$, it is readily seen that
	\begin{equation*}
		\lim_{\epsilon,\delta,\lambda\to 0}\|H_{\epsilon,\delta,\lambda}\|_{L_x^{p_0}L_v^{q_0}} = \|f\mathds{1}_K(v)\|_{L_x^{p_0}L_v^{q_0}}
	\end{equation*}
	and
	\begin{equation*}
		\lim_{\epsilon,\delta,\lambda\to 0}
		\left\|\chi\left(\epsilon x\right)
		\chi\left(\delta \pi_v x\right)
		\frac{v\cdot\nabla_xf}{(1+\lambda^2 f^2)^\frac 32}\mathds{1}_K(v)\right\|_{L_x^{p_1}L_v^{q_1}}
		= \|v\cdot\nabla_x f\mathds{1}_K(v)\|_{L_x^{p_1}L_v^{q_1}}.
	\end{equation*}
	Therefore, noticing that
	\begin{equation*}
		\begin{aligned}
			v\cdot\nabla_x H_{\epsilon,\delta,\lambda}(x,v)&=\epsilon(v\cdot\nabla\chi)\left(\epsilon x\right)
			\chi\left(\delta \pi_v x\right)
			\frac{f(x,v)}{(1+\lambda^2 f(x,v)^2)^\frac 12}\mathds{1}_K(v)
			\\
			&\quad+\chi\left(\epsilon x\right)
			\chi\left(\delta \pi_v x\right)
			\frac{v\cdot\nabla_xf(x,v)}{(1+\lambda^2 f(x,v)^2)^\frac 32}\mathds{1}_K(v),
		\end{aligned}
	\end{equation*}
	there only remains to show
	\begin{equation}\label{limit:1}
		\lim_{\epsilon,\delta,\lambda\to 0}
		\left\|\epsilon(v\cdot\nabla\chi)\left(\epsilon x\right)
		\chi\left(\delta \pi_v x\right)
		\frac{f}{(1+\lambda^2 f^2)^\frac 12}\mathds{1}_K(v)\right\|_{L_x^{p_1}L_v^{q_1}}
		= 0
	\end{equation}
	provided the approximation parameters $\epsilon$, $\delta$ and $\lambda$ are suitably chosen as they vanish simultaneously.
	
	To this end, denoting the angle between $x$ and $v$ by $\theta\in[0,\pi]$, so that $|\pi_vx|=|x|\sin\theta$, and various independent constants by $C>0$, observe that
	\begin{equation*}
		\begin{aligned}
			\left\|\epsilon(v\cdot\nabla\chi)\left(\epsilon x\right)
			\chi\left(\delta \pi_v x\right)
			\frac{f}{(1+\lambda^2 f^2)^\frac 12}\mathds{1}_K(v)\right\|_{L_x^{p_1}L_v^{q_1}}
			\hspace{-40mm}&
			\\
			&\leq
			\frac{\epsilon}{\lambda}\|\nabla\chi\|_{L^\infty}\big\||v|\mathds{1}_K(v)\big\|_{L^\infty}
			\left\|\mathds{1}_{\{1\leq\epsilon|x|\leq 2\}}
			\left\|\mathds{1}_{\{\delta|x|\sin\theta\leq 2\}}\mathds{1}_K(v)\right\|_{L_v^{q_1}}
			\right\|_{L_x^{p_1}}
			\\
			&\leq
			C\frac{\epsilon}{\lambda}
			\left\|
			\frac{\mathds{1}_{\{1\leq\epsilon|x|\leq 2\}}}{(\delta|x|)^\frac{n-1}{q_1}}
			\right\|_{L_x^{p_1}}
			\leq C\frac{\epsilon^{1+\frac{n-1}{q_1}-\frac n{p_1}}}{\lambda\delta^{\frac{n-1}{q_1}}}.
		\end{aligned}
	\end{equation*}
	Then, noticing that \eqref{constraint:2} implies $1+\frac{n-1}{q_1}-\frac n{p_1}>0$, we conclude that it is always possible to choose $\epsilon$, $\delta$ and $\lambda$ (e.g., take $\delta=\lambda=|\log\epsilon|^{-1}$) such that
	\begin{equation*}
		\lim_{\epsilon,\delta,\lambda\to 0}\frac{\epsilon^{1+\frac{n-1}{q_1}-\frac n{p_1}}}{\lambda\delta^{\frac{n-1}{q_1}}}=0,
	\end{equation*}
	thereby establishing the limit \eqref{limit:1} and, thus, completing the proof of the lemma.
\end{proof}

\begin{rema}
	The problem of extending the velocity averaging estimate \eqref{general:1} for smooth compactly supported functions to more general functional settings by approximation procedures becomes much easier to settle when one replaces the space $L^{p_1}_xL^{q_1}_v$ by $L^{q_1}_vL^{p_1}_x$ in the right-hand side of \eqref{general:1} or \eqref{general:3}. Indeed, in this case, it is possible to adapt, with minor changes, the preceding proof without requiring the constraint \eqref{constraint:2}.
	
	More precisely, in order to perform such an adaptation, the only delicate step consists in verifying that the limit \eqref{limit:1} still holds when one replaces the $L^{p_1}_xL^{q_1}_v$-norm by an $L^{q_1}_vL^{p_1}_x$-norm, for any $1<p_1,q_1<\infty$ and an appropriate choice of parameters $\epsilon$, $\delta$ and $\lambda$. This follows from the estimate
	\begin{equation*}
		\begin{aligned}
			\left\|\epsilon(v\cdot\nabla\chi)\left(\epsilon x\right)
			\chi\left(\delta \pi_v x\right)
			\frac{f}{(1+\lambda^2 f^2)^\frac 12}\mathds{1}_K(v)\right\|_{L_v^{q_1}L_x^{p_1}}
			\hspace{-60mm}&
			\\
			&\leq
			\frac{\epsilon}{\lambda}\|\nabla\chi\|_{L^\infty}\big\||v|\mathds{1}_K(v)\big\|_{L^\infty}
			\left\|
			\left\|\mathds{1}_{\{1\leq\epsilon|x|\leq 2\}}\mathds{1}_{\{\delta|x|\sin\theta\leq 2\}}\right\|_{L_x^{p_1}}
			\mathds{1}_K(v)\right\|_{L_v^{q_1}}
			\leq C\frac{\epsilon^{1-\frac 1{p_1}}}{\lambda\delta^{\frac{n-1}{p_1}}},
		\end{aligned}
	\end{equation*}
	whereby, setting $\delta=\lambda=|\log\epsilon|^{-1}$,
	\begin{equation*}
		\lim_{\epsilon,\delta,\lambda\to 0}
		\left\|\epsilon(v\cdot\nabla\chi)\left(\epsilon x\right)
		\chi\left(\delta \pi_v x\right)
		\frac{f}{(1+\lambda^2 f^2)^\frac 12}\mathds{1}_K(v)\right\|_{L_v^{q_1}L_x^{p_1}}
		= 0
	\end{equation*}
	because $p_1>1$.
\end{rema}

The next approximation lemma is only useful when neither \eqref{constraint:1} nor \eqref{constraint:2} hold.

\begin{lem}\label{approximation:5}
	Consider any set of parameters $1< p_0,p_1,q_0,q_1,r<\infty$ and $s\geq 0$. Let us assume that, for any compact set $K\subset\mathbb{R}^n$, there exists a finite constant $C_K>0$ such that the estimate \eqref{general:1}, or \eqref{general:3}, holds for any $f(x,v)\in C_c^\infty(\mathbb{R}^n\times\mathbb{R}^n)$. Then \eqref{general:1}, respectively \eqref{general:3}, remains valid for any compactly supported $f(x,v)$ such that
	\begin{equation*}
		\begin{aligned}
			f&\in L^{p_0}(\mathbb{R}^n_x;L^{q_0}(\mathbb{R}^n_v)),
			\\
			v\cdot\nabla_x f&\in L^{p_1}(\mathbb{R}^n_x;L^{q_1}(\mathbb{R}^n_v)).
		\end{aligned}
	\end{equation*}
\end{lem}

\begin{rema}
	Observe that, since \eqref{general:1} and \eqref{general:3} can be localized in $v$, it is sufficient to assume that $f(x,v)$ is compactly supported in $x$ only.
\end{rema}

\begin{proof}
	We utilize here the renormalizations
	\begin{equation*}
		h_\lambda(x,v)=\frac{f(x,v)}{(1+\lambda^2 f(x,v)^2)^\frac 12},
	\end{equation*}
	where $0<\lambda\leq 1$. In particular, one has that
	\begin{equation*}
		h_\lambda\to f
		\qquad\text{in }L^{p_0}_xL^{q_0}_v
	\end{equation*}
	and
	\begin{equation*}
		\|v\cdot\nabla_xh_\lambda\|_{L^{p_1}_xL^{q_1}_v}
		=\left\|\frac{v\cdot\nabla_xf}{(1+\lambda^2 f^2)^\frac 32}\right\|_{L^{p_1}_xL^{q_1}_v}
		\leq\|v\cdot\nabla_xf\|_{L^{p_1}_xL^{q_1}_v}.
	\end{equation*}
	
	Now, since each $h_\lambda(x,v)$ is bounded pointwise and compactly supported, an application of Lemma~\ref{approximation:2} yields that \eqref{general:1} holds for $h_\lambda$. It follows that
	\begin{equation*}
		\begin{aligned}
			\left\|\int_{K}fdv\right\|_{\dot W^{s,r}(\mathbb{R}^n_x)}
			&\leq\liminf_{\lambda\to 0}
			\left\|\int_{K}h_\lambda dv\right\|_{\dot W^{s,r}(\mathbb{R}^n_x)}
			\\
			&\leq \liminf_{\lambda\to 0} C_K
			\Big(\|h_\lambda\|_{L^{p_0}(\mathbb{R}^n_x;L^{q_0}(\mathbb{R}^n_v))}
			+\|v\cdot\nabla_xh_\lambda\|_{L^{p_1}(\mathbb{R}^n_x;L^{q_1}(\mathbb{R}^n_v))}\Big)
			\\
			&\leq C_K
			\Big(\|f\|_{L^{p_0}(\mathbb{R}^n_x;L^{q_0}(\mathbb{R}^n_v))}
			+\|v\cdot\nabla_xf\|_{L^{p_1}(\mathbb{R}^n_x;L^{q_1}(\mathbb{R}^n_v))}\Big),
		\end{aligned}
	\end{equation*}
	which concludes the proof.
\end{proof}

Finally, we provide an approximation lemma for velocity averaging estimates set in Sobolev spaces that are based on mixed Lebesgue norms. Note that, in this setting, it is in general not possible to use renormalization techniques so easily.

\begin{lem}\label{approximation:6}
	Consider any set of parameters $1< p_0,p_1,q_0,q_1,r<\infty$ and $a,b,\alpha,\beta,\sigma\in\mathbb{R}$. Let us assume that, for any compact set $K\subset\mathbb{R}^n$, there exists a finite constant $C_K>0$ such that the estimate
	\begin{equation}\label{general:4}
		\begin{aligned}
			\left\|\int_{K}fdv\right\|_{W^{\sigma,r}(\mathbb{R}^n_x)} & \leq C_K
			\Big(\|(1-\Delta_x)^\frac a 2(1-\Delta_v)^\frac \alpha 2f\|_{L^{p_0}(\mathbb{R}^n_x;L^{q_0}(\mathbb{R}^n_v))}
			\\
			&\quad +\|(1-\Delta_x)^\frac b 2(1-\Delta_v)^\frac \beta 2v\cdot\nabla_xf\|_{L^{p_1}(\mathbb{R}^n_x;L^{q_1}(\mathbb{R}^n_v))}\Big)
		\end{aligned}
	\end{equation}
	holds for any $f(x,v)\in C_c^\infty(\mathbb{R}^n\times\mathbb{R}^n)$ vanishing for $v$ outside $K$. Then, by possibly increasing the value of $C_K$, the above estimate remains valid for any compactly supported $f(x,v)$ vanishing for $v$ outside $K$ such that
	\begin{equation}\label{general:5}
		\begin{aligned}
			(1-\Delta_x)^\frac a 2(1-\Delta_v)^\frac \alpha 2f&\in L^{p_0}(\mathbb{R}^n_x;L^{q_0}(\mathbb{R}^n_v)),
			\\
			(1-\Delta_x)^\frac b 2(1-\Delta_v)^\frac \beta 2f&\in
			L^{p_1}(\mathbb{R}^n_x;L^{q_1}(\mathbb{R}^n_v)),
			\\
			(1-\Delta_x)^\frac b 2(1-\Delta_v)^\frac \beta 2
			v\cdot\nabla_x f&\in L^{p_1}(\mathbb{R}^n_x;L^{q_1}(\mathbb{R}^n_v)).
		\end{aligned}
	\end{equation}
\end{lem}

\begin{rema}
	It is readily seen from the proof below that Lemma \ref{approximation:6} also holds if one replaces \eqref{general:4} by the more general estimate
	\begin{equation*}
		\begin{aligned}
			\left\|\int_{K}fdv\right\|_{W^{\sigma,r}}^2 & \leq C_K
			\Big(\|(1-\Delta_x)^\frac a 2(1-\Delta_v)^\frac \alpha 2f\|_{L^{p_0}_xL^{q_0}_v}
			\|(1-\Delta_x)^\frac b 2(1-\Delta_v)^\frac \beta 2v\cdot\nabla_xf\|_{L^{p_1}_xL^{q_1}_v}
			\\
			&\quad +\|(1-\Delta_x)^\frac c 2(1-\Delta_v)^\frac \gamma 2f\|_{L^{p_2}_xL^{q_2}_v}\Big),
		\end{aligned}
	\end{equation*}
	where $1< p_0,p_1,p_2,q_0,q_1,q_2,r<\infty$ and $a,b,c,\alpha,\beta,\gamma,\sigma\in\mathbb{R}$,
	and  \eqref{general:5} by the bounds
	\begin{equation*}
		\begin{aligned}
			(1-\Delta_x)^\frac a 2(1-\Delta_v)^\frac \alpha 2f&\in L^{p_0}(\mathbb{R}^n_x;L^{q_0}(\mathbb{R}^n_v)),
			\\
			(1-\Delta_x)^\frac b 2(1-\Delta_v)^\frac \beta 2f&\in
			L^{p_1}(\mathbb{R}^n_x;L^{q_1}(\mathbb{R}^n_v)),
			\\
			(1-\Delta_x)^\frac c 2(1-\Delta_v)^\frac \gamma 2f&\in
			L^{p_2}(\mathbb{R}^n_x;L^{q_2}(\mathbb{R}^n_v)),
			\\
			(1-\Delta_x)^\frac b 2(1-\Delta_v)^\frac \beta 2
			v\cdot\nabla_x f&\in L^{p_1}(\mathbb{R}^n_x;L^{q_1}(\mathbb{R}^n_v)).
		\end{aligned}
	\end{equation*}
\end{rema}

\begin{proof}
	We follow a strategy similar to the proof of Lemma \ref{approximation:2}. Thus, we start by considering any compactly supported $f(x,v)$ satisfying the bounds \eqref{general:5}. Then, taking some cutoff function $\phi(x,v)\in C_c^\infty(\mathbb{R}^n\times\mathbb{R}^n)$ and defining
	\begin{equation*}
		\phi_{\epsilon}(x,v)=\epsilon^{-2n}\phi(x/\epsilon,v/\epsilon),
	\end{equation*}
	with $0<\epsilon\leq 1$, a direct commutator estimate, \emph{\`a la Friedrichs,} yields that
	\begin{equation*}
		\begin{aligned}
			\left\|(1-\Delta_x)^\frac b 2(1-\Delta_v)^\frac \beta 2
			\big(\phi_{\epsilon}*(v\cdot\nabla_x f)
			-v\cdot\nabla_x(\phi_{\epsilon}* f)\big)\right\|_{L^{p_1}_xL^{q_1}_v}\hspace{-60mm}&
			\\
			&=
			\left\|(1-\Delta_x)^\frac b 2(1-\Delta_v)^\frac \beta 2(v\cdot\nabla_x \phi_{\epsilon})*f\right\|_{L^{p_1}_xL^{q_1}_v}
			\\
			&=
			\left\|(v\cdot\nabla_x \phi_{\epsilon})*(1-\Delta_x)^\frac b 2(1-\Delta_v)^\frac \beta 2f\right\|_{L^{p_1}_xL^{q_1}_v}
			\\
			&\leq\|v\cdot\nabla_x\phi\|_{L^1_{x,v}}
			\left\|(1-\Delta_x)^\frac b 2(1-\Delta_v)^\frac \beta 2f\right\|_{L^{p_1}_xL^{q_1}_v}.
		\end{aligned}
	\end{equation*}
	In particular, since the Schwartz space $\mathcal{S}$ is stable under the action of the differential operator $(1-\Delta_x)^\frac b 2(1-\Delta_v)^\frac \beta 2$ and is dense in $L^{p_1}L^{q_1}$, we obtain, further assuming $\int_{\mathbb{R}^{2n}}\phi(x,v) dxdv=1$, that
	\begin{equation*}
		\lim_{\epsilon\to 0}
		\left\|(1-\Delta_x)^\frac b 2(1-\Delta_v)^\frac \beta 2
		\big(\phi_{\epsilon}*(v\cdot\nabla_x f)
		-v\cdot\nabla_x(\phi_{\epsilon}* f)\big)\right\|_{L^{p_1}_xL^{q_1}_v}
		=0,
	\end{equation*}
	whence
	\begin{equation*}
		\lim_{\epsilon\to 0}
		\left\|(1-\Delta_x)^\frac b 2(1-\Delta_v)^\frac \beta 2
		\big(v\cdot\nabla_x(\phi_{\epsilon}* f)\big)\right\|_{L^{p_1}_xL^{q_1}_v}
		=\left\|(1-\Delta_x)^\frac b 2(1-\Delta_v)^\frac \beta 2
		\big(v\cdot\nabla_xf\big)\right\|_{L^{p_1}_xL^{q_1}_v}.
	\end{equation*}
	Similarly, it holds that
	\begin{equation*}
		\lim_{\epsilon\to 0}
		\left\|(1-\Delta_x)^\frac a 2(1-\Delta_v)^\frac \alpha 2
		(\phi_{\epsilon}* f)\right\|_{L^{p_0}_xL^{q_0}_v}
		=\left\|(1-\Delta_x)^\frac a 2(1-\Delta_v)^\frac \alpha 2
		f\right\|_{L^{p_0}_xL^{q_0}_v}.
	\end{equation*}
	
	Now, since $f_\epsilon=\phi_\epsilon* f$ is compactly supported and smooth,
	one deduces from \eqref{general:4} that
	\begin{equation*}
		\begin{aligned}
			\left\|\int_{\widetilde K}f_\epsilon dv\right\|_{W^{\sigma,r}} & \leq C_{\widetilde K}
			\Big(\|(1-\Delta_x)^\frac a 2(1-\Delta_v)^\frac \alpha 2f_\epsilon\|_{L^{p_0}_xL^{q_0}_v}
			\\
			&\quad +\|(1-\Delta_x)^\frac b 2(1-\Delta_v)^\frac \beta 2v\cdot\nabla_xf_\epsilon\|_{L^{p_1}_xL^{q_1}_v}\Big),
		\end{aligned}
	\end{equation*}
	where $\widetilde K\subset\mathbb{R}^n$ is a compact set containing the support in $v$ of $f_\epsilon$ for all $0<\epsilon\leq 1$.
	Finally, letting $\epsilon$ tend to zero, we infer, by weak lower-semicontinuity of norms, that
	\begin{equation*}
		\begin{aligned}
			\left\|\int_{K}f dv\right\|_{W^{\sigma,r}}
			& \leq
			\liminf_{\epsilon\to0}\left\|\int_{\widetilde K}f_\epsilon dv\right\|_{W^{\sigma,r}}
			\\
			& \leq C_{\widetilde K}
			\Big(\|(1-\Delta_x)^\frac a 2(1-\Delta_v)^\frac \alpha 2f\|_{L^{p_0}_xL^{q_0}_v}
			\\
			&\quad +\|(1-\Delta_x)^\frac b 2(1-\Delta_v)^\frac \beta 2v\cdot\nabla_xf\|_{L^{p_1}_xL^{q_1}_v}\Big),
		\end{aligned}
	\end{equation*}
	which establishes that \eqref{general:4} holds for every compactly supported $f$ satisfying \eqref{general:5}.
\end{proof}

% \nocite*
\bibliographystyle{plain} 
\bibliography{energy}

\end{document}